\documentclass[a4paper,11pt]{article}

\usepackage[T1]{fontenc}
\usepackage{lmodern}

\usepackage{dsfont}

\usepackage[latin1]{inputenc}
\usepackage{amsmath}
\usepackage{amsthm}
\usepackage{amssymb}
\usepackage{mathrsfs}
\usepackage{graphicx}
\usepackage[all]{xy}
\usepackage{hyperref}

\usepackage{abstract} 

\newtheorem{thm}{Theorem}[section]
\newtheorem{cor}[thm]{Corollary}
\newtheorem{claim}[thm]{Claim}
\newtheorem{fact}[thm]{Fact}

\newtheorem{lemma}[thm]{Lemma}
\newtheorem{prop}[thm]{Proposition}

\theoremstyle{definition}
\newtheorem{definition}[thm]{Definition}

\newtheorem{remark}[thm]{Remark}

\def\rquotient#1#2{%
	\makeatletter
	\raise.3ex\hbox{$#1$}/\lower.3ex\hbox{$#2$}%
	\makeatother
}	

\makeatletter
\newcommand{\subjclass}[2][2010]{%
	\let\@oldtitle\@title%
	\gdef\@title{\@oldtitle\footnotetext{#1 \emph{Mathematics subject classification.} #2}}%
}
\newcommand{\keywords}[1]{%
	\let\@@oldtitle\@title%
	\gdef\@title{\@@oldtitle\footnotetext{\emph{Key words and phrases.} #1.}}%
}
\makeatother

\newcommand{\Address}{{
		\bigskip
		\small
		
		\textsc{Institut Montpellierain Alexander Grothendieck, 499-554 Rue du Truel, 34090 Montpellier, France.}\par\nopagebreak
		\textit{E-mail address}: \texttt{anthony.genevois@umontpellier.fr}
		
}}

\title{Lampligther groups, median spaces, and Hilbertian geometry}
\date{\today}
\author{Anthony Genevois}


\begin{document}

\maketitle

\begin{abstract}
From any two median spaces $X,Y$, we construct a new median space $X \circledast Y$, referred to as the \emph{diadem product} of $X$ and $Y$, and we show that this construction is compatible with wreath products in the following sense: given two finitely generated groups $G,H$ and two (equivariant) coarse embeddings into median spaces $X,Y$, there exist a(n equivariant) coarse embedding $G\wr H \to X \circledast Y$. As an application, we prove that 
$$\alpha_1(G \wr H) \geq \min(\alpha_1(G),\alpha_1(H))/2 \text{ for all finitely generated groups $G,H$,}$$
where $\alpha_1(\cdot)$ denotes the $\ell^1$-compression. As an other consequence, we recover several well-known theorems related to the Hilbertian geometry of wreath products from a unified point of view: the characterisation of wreath products satisfying Kazhdan's property (T) or the Haagerup property, as well as their discrete versions (FW) and (PW).
\end{abstract}

\tableofcontents

\section{Introduction}

\noindent
Recall that, given two groups $G$ and $H$, the \emph{wreath product} $G \wr H$ is defined as the semidirect product $\left( \bigoplus_H G \right) \rtimes H$ where $H$ acts on the direct sum by permuting the coordinates. These groups are also called \emph{lamplighter groups}, a terminology coined by Jim Cannon (see \cite{MR1062874}). The family of lamplighter groups is well-known in group theory, and has been studied from various perspectives over the years. On the one hand, lamplighter groups have an easy and explicit definition, allowing an easy access to various properties and calculations. On the other hand, these groups are sufficiently exotic, i.e. sufficiently far away from most of the well-understood classes of groups exhibited in the literature, in order to exhibit interesting behaviours. The combination of these two observations probably explains the success of lamplighter groups, and why they are often used to produce counterexamples. 

\medskip \noindent
In this article, we are interested in the $\ell^1$-geometry of wreath products. How taking the wreath product of two finitely generated groups can affect the compatibility between the geometry of the group and the geometry of an $\ell^1$-space? In fact, because $\ell^1$-spaces are median spaces and that, conversely, median spaces isometrically embed into $\ell^1$-spaces \cite{medianviewpoint}, we can alternatively ask the previous question for median spaces. Recall that:

\begin{definition}
Let $X$ be a metric space. Given two points $x,y \in X$, the \emph{interval} between $x$ and $y$ is
$$I(x,y) = \{ z \in X \mid d(x,y)=d(x,z)+d(z,y) \}.$$
Given any three points $x,y,z \in X$, a point in the intersection $I(x,y) \cap I(y,z) \cap I(x,z)$ is a \emph{median point} of $x$, $y$ and $z$. The space $X$ is \emph{median} if any triple of points admits a unique median.
\end{definition}

\noindent
In this article, we use the point of view offered by median spaces. Our main goal is to transfer the wreath product between groups to an operation between median spaces that create another median space.

\begin{definition}
Let $X,Y$ be two median spaces and $1 \in X$ a basepoint. The \emph{diadem product} $(X,1)\circledast Y$  (or simply $X \circledast Y$) is the set of \emph{wreaths} $(C,\varphi)$, where 
\begin{itemize}
	\item $C$ is a convex subspace of $Y$ that is \emph{finitely generated} (i.e. the convex hull of finitely many points);
	\item $\varphi : Y \to X$ satisfies $\varphi(y)=1$ for all but finitely many $y \in Y$ (written $\varphi \in X^{(Y)}$), 
\end{itemize}
endowed with the metric $\delta$ defined as
$$((C_1,\varphi_1),(C_2,\varphi_2)) \mapsto 2 \cdot \mu(C_1 \cup C_2 \cup \varphi_1 \Delta \varphi_2)- \mu(C_1)- \mu(C_2)+ \sum\limits_{y \in Y} d(\varphi_1(y), \varphi_2(y))$$
where $\varphi_1 \Delta \varphi_2$ denotes the set of points where $\varphi_1,\varphi_2$ differ and where $\mu(\cdot)$ denotes the measure of the collection of hyperplanes crossing the subspace under consideration (see Section \ref{section:medianspaces}).
\end{definition}

\noindent
We refer to Section~\ref{section:warmup} for an illustration of diadem products in a simple case. Our definition of diadem products is inspired by \cite[Section 9]{Qm}, where, given two groups $G,H$ with $H$ acting on a median graph, we constructed an action of $G \wr H$ on a quasi-median graph and deduced estimations on the $\ell^2$-compression.

\medskip \noindent
The general idea is that, given two finitely generated groups $G,H$, to any two maps $\Phi,\Psi$ from our groups to median spaces $X,Y$ can be associated a new map $\Phi \wr \Psi : G \wr H \to X \circledast Y$ in such that a way the amount of geometry preserved by $\Phi \wr \Psi$ is directly related to the amount of geometry preserved by $\Phi$ and $\Psi$. We motivate this idea from two points of view.

\medskip \noindent
Our first point of view is purely geometric: the wreath product of two finitely generated groups that embed nicely in median spaces also embeds nicely in some median space. In fact, instead of finitely generated groups, we can express our results for arbitrary graphs thanks to the following definition:

\begin{definition}
Let $G,H$ be two graphs and $1 \in G$ a basepoint. The \emph{wreath product} $(G,1) \wr H$ (or simply $G \wr H$) is the graph whose vertices are the pairs $(\varphi,h)$ where $h \in H$ and where $\varphi : H \to G$ satisfies $\varphi(k)=1$ for all but finitely many $k \in H$ (written $\varphi  \in G^{(H)}$), and whose edges link two vertices $(\varphi_1,y_1), (\varphi_2,y_2)$ if either $\varphi_1=\varphi_2$ and $y_1,y_2$ are adjacent in $H$ or $y_1=y_2$ and $\varphi_1,\varphi_2$ only differ at $y_1=y_2$ with $\varphi(y_1),\varphi(y_2)$ adjacent in $G$. 
\end{definition}

\noindent
Observe that, given two groups $G,H$ and two generating sets $R \subset G, S \subset H$, we have
$$\mathrm{Cayl}(G \wr H, R\cup S)= (\mathrm{Cayl}(G,R),1) \wr \mathrm{Cayl}(H,S),$$ 
justifying our terminology.

\noindent
Our first application of diadem products is that the wreath product of two (uniformly locally finite) graphs that coarsely embed in $\ell^1$-spaces also embeds in some $\ell^1$-space, or equivalently:

\begin{thm}\label{thm:WreathCoarse}
Let $G,H$ be two graphs with $H$ uniformly locally finite. Then $G \wr H$ coarsely embeds in a Hilbert space if and only if so do $G,H$.
\end{thm}

\noindent
The property of being coarsely embeddable in some Hilbert space has been popularised by Yu in \cite{MR1728880}, where it is proved that a finitely generated group that coarsely embeds in some Hilbert space satisfies the famous Novikov conjecture. 

\medskip \noindent
Regarding Theorem \ref{thm:WreathCoarse}, it is natural to ask whether a control on the metric distortion is possible. Recall that, given a Lipschitz map $f : R \to S$ between two metric spaces, one says that $f$ \emph{has compression $\geq \alpha$} if there exists some constant $C >0$ such that
$$d(f(a),f(b)) \geq C \cdot d(a,b)^\alpha \text{ for all $a,b \in R$}.$$
The \emph{$\ell^p$-compression} of $R$, denote by $\alpha_p(R)$, is the supremum of the $\alpha$ such that there exists a Lipschitz map of compression $\geq \alpha$ from $R$ to an $\ell^p$-space. Roughly speaking, the $\ell^p$-compression of a metric space is a real number between zero and one that quantifies the compatibility between the geometry of the space and the geometry of an $\ell^p$-space. The first examples of finitely generated groups with $\ell^2$-compression in $(0,1)$, namely Thompson's group $F$ and the lamplighter group $\mathbb{Z} \wr \mathbb{Z}$, are exhibited in \cite{MR2271228}. Since then, $\ell^p$-compressions of wreath products have received a lot of attention (see for instance \cite{lamplighterSV, Haagerupwreath, MR2803851, MR2439557, MR2783928, MR2644886, MR2439428, MR4199729}). 

\medskip \noindent
A quantitative version of Theorem \ref{thm:WreathCoarse} leads to the following statement:

\begin{thm}\label{thm:Compression}
Let $G,H$ be two graphs with $H$ uniformly locally finite. Then $$\alpha_1(G \wr H) \geq \frac{1}{2} \cdot \min(\alpha_1(G),\alpha_1(H)).$$ 
\end{thm}

\noindent
This estimates improves the lower bound given by \cite[Theorem 1.1]{MR2644886} for $p=1$. We emphasize that having $\ell^p$-compression one does not imply that the metric space under consideration admits a biLipschitz embedding in some $\ell^p$-space. For instance, a finitely generated free group has $\ell^2$-compression one but it does not admit a biLipschitz embedding in some Hilbert space \cite{MR880292}. Several wreath products are known to admit biLipschitz embedding in $\ell^1$-spaces, such that $\mathbb{Z}_2 \wr \mathbb{Z}$ \cite{MR2439557} and $\mathbb{Z}_2 \wr \mathbb{F}$ \cite{Haagerupwreath} (see also \cite{GeometryLamp}), but the problem is difficult in general. For instance, in \cite{MR2783928}, the authors show that $\mathbb{Z}_2 \wr\mathbb{Z}^2$ has $\ell^1$-compression one but leave the existence of a biLipschitz embedding as an open question. As an application of our diadem products, we prove that:

\begin{thm}\label{thm:Hyperbolic}
Let $G,H$ be two graphs. Assume that $H$ is a uniformly locally finite median hyperbolic graph. If $G$ biLipschitz embeds into an $\ell^1$-space, then so does $G \wr H$.
\end{thm}

\noindent
For instance, the theorem applies to the groups $\mathbb{Z}_n \wr \mathbb{Z}$, $\mathbb{Z}^n \wr \mathbb{Z}$, $\mathbb{F}_n \wr \mathbb{Z}$, $\mathbb{Z}_n \wr\mathbb{F}_r$, $\mathbb{Z}^n \wr\mathbb{F}_r$ and $\mathbb{F}_n \wr \mathbb{F}_r$. In fact, in all these cases the median spaces are discrete, so our construction provides a biLipschitz embedding in an infinite Hamming cube $\{0,1\}^{(\mathbb{N})}$. 

\medskip \noindent
Our second point of view dynamical: the wreath product of two groups that act nicely on $\ell^1$-spaces also acts nicely on some $\ell^1$-space. Such results are obtained by noticing that, if two groups act on two median spaces, then their wreath product acts on the diadem product of the corresponding spaces. In view of the characterisation of Kazhdan's property (T) and a-T-menability provided by \cite{medianviewpoint}, we recover the two following known statements:

\begin{thm}\label{thm:TandHaagerup}
Let $G,H$ be two non-trivial discrete groups. 
\begin{itemize}
	\item \emph{\cite{measuredwalls}} $G\wr H$ has property (T) if and only if $H$ is finite and $G$ has property (T). 
	\item \emph{\cite{Haagerupwreath}} $G\wr H$ is a-T-menable if and only if so are $G$ and $H$.
\end{itemize}
\end{thm}

\noindent
Because a discrete group has property (T) if and only if it cannot act on a median space with unbounded orbits, a natural discrete analogue is the property (FW), asking that no action of the group on a median graph can have unbounded orbits. Similarly, being a-T-menable amounts to admitting a metrically proper action on a median space, and the corresponding discrete version of it, namely the property (PW), requires the existence of a metrically proper action on a median graph. By noticing that the diadem product of two median graphs produces a median graphs, we deduce a discrete analogue of Theorem \ref{thm:TandHaagerup}:

\begin{thm}\label{thm:FWandPW}
Let $G,H$ be two non-trivial groups. 
\begin{itemize}
	\item \emph{\cite{TandFWforWreath}} $G\wr H$ has property (FW) if and only if $H$ is finite and $G$ have property~(FW). 
	\item \emph{\cite{Haagerupwreath}} $G\wr H$ has property (PW) if and only if so do $G$ and $H$.
\end{itemize}
\end{thm}

\noindent
See also \cite{Qm} for another proof of the second point.

\paragraph{Acknowledgments.} I am grateful to Elia Fioravanti for interesting discussions about median spaces; to Victor Chepo\"{i}, for having indicated to me the reference \cite{vantheory}; and to Bruno Duchesne and J\'er\'emie Brieussel for their comments on a previous version of my manuscript.

\section{Warm up}\label{section:warmup}

\noindent
In this section, we sketch a proof of the fact that the wreath product $\mathbb{Z} \wr \mathbb{Z}^2$ acts metrically properly on a median graph, in order to motivate the definitions used in the next section. 

\medskip \noindent
An element of the wreath product $\mathbb{Z} \wr \mathbb{Z}^2$, thought of as a lamplighter group, can be described by an infinite grid whose vertices are labelled by integers, such that all but finitely many vertices are labelled by $0$, together with an arrow pointing to some vertex. See Figure \ref{figure1}. Formally, the labelled grid encodes the coordinate along $\bigoplus\limits_{p \in \mathbb{Z}^2} \mathbb{Z}$ and the arrow the coordinate along $\mathbb{Z}^2$. Moreover, $\mathbb{Z} \wr \mathbb{Z}^2$ has a natural generating set such that right-multiplying an element of $\mathbb{Z} \wr \mathbb{Z}^2$  by one of these generators corresponds to modifying the integer of the vertex where the arrow is (by adding $\pm 1$) or to moving the arrow to an adjacent vertex. 
\begin{figure}
\begin{center}
\includegraphics[trim={0 10cm 17cm 0},clip,scale=0.25]{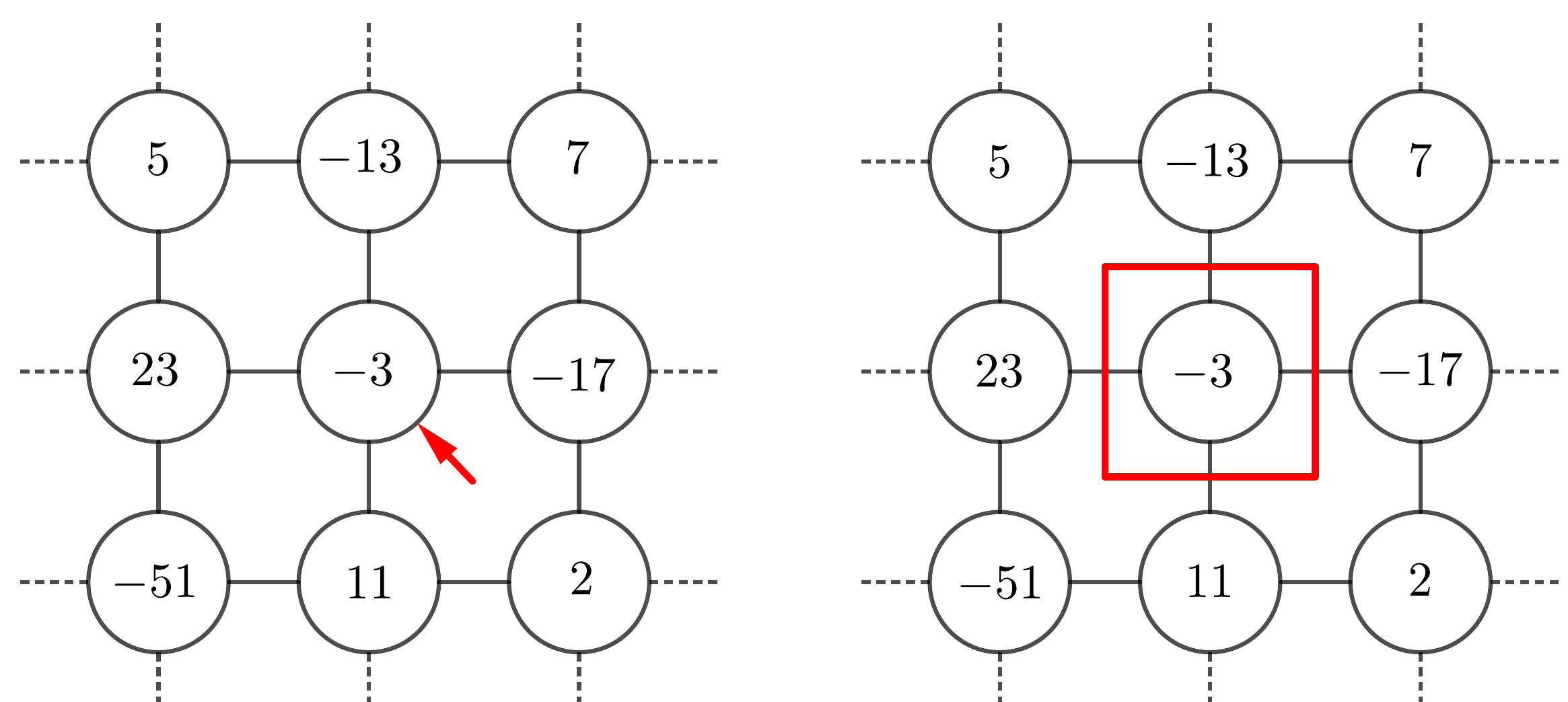}
\end{center}
\caption{Replacing the arrow with a rectangle.}
\label{figure1}
\end{figure}

\medskip \noindent
Essentially, our construction lies on the following idea: replace the arrow of the previous description with a rectangle (whose corners have their coordinates in $\frac{1}{2} \mathbb{Z}$) containing a single vertex of the grid (see Figure \ref{figure1}), and, instead of moving the arrow from one vertex to an adjacent vertex, move the sides of the rectangle independently. For instance, in order to move the rectangle to one vertex to an adjacent vertex, three moves are necessary; see Figure \ref{figure2}. More formally, we define a \emph{wreath} as the data $(R, \varphi)$ of a rectangle $R$ and a map $\varphi : \mathbb{Z}^2 \to \mathbb{Z}$ with finite support. Now, our elementary moves on a given wreath $(R, \varphi)$ are the followings: modify the integer of a vertex which belongs to (the interior of) $R$ by adding $\pm 1$, or translate one (and only one) side of $R$ by a unit vector. Among the wreaths, we recover the group $\mathbb{Z} \wr \mathbb{Z}^2$ as the wreaths whose rectangles contain a single vertex of the grid. Moreover, we have a natural action of $\mathbb{Z} \wr \mathbb{Z}^2$ on the set of wreaths extending the left-multiplication:
$$( p, \psi) \cdot (R, \varphi) = \left( R+p, \psi(\cdot)+ \varphi(\cdot -p) \right).$$
Now, define the \emph{graph of wreaths} $\mathfrak{W}$ (which will correspond to the diadem product of the two median graphs $\mathbb{Z}$ and $\mathbb{Z}^2$) as the graph whose vertices are the wreaths and whose edges link two wreaths such that one can be obtained from another by an elementary move. We claim that $\mathfrak{W}$ is a median graph on which $\mathbb{Z} \wr \mathbb{Z}^2$ acts metrically properly. 
\begin{figure}
\begin{center}
\includegraphics[trim={0 10cm 12cm 0},clip,scale=0.25]{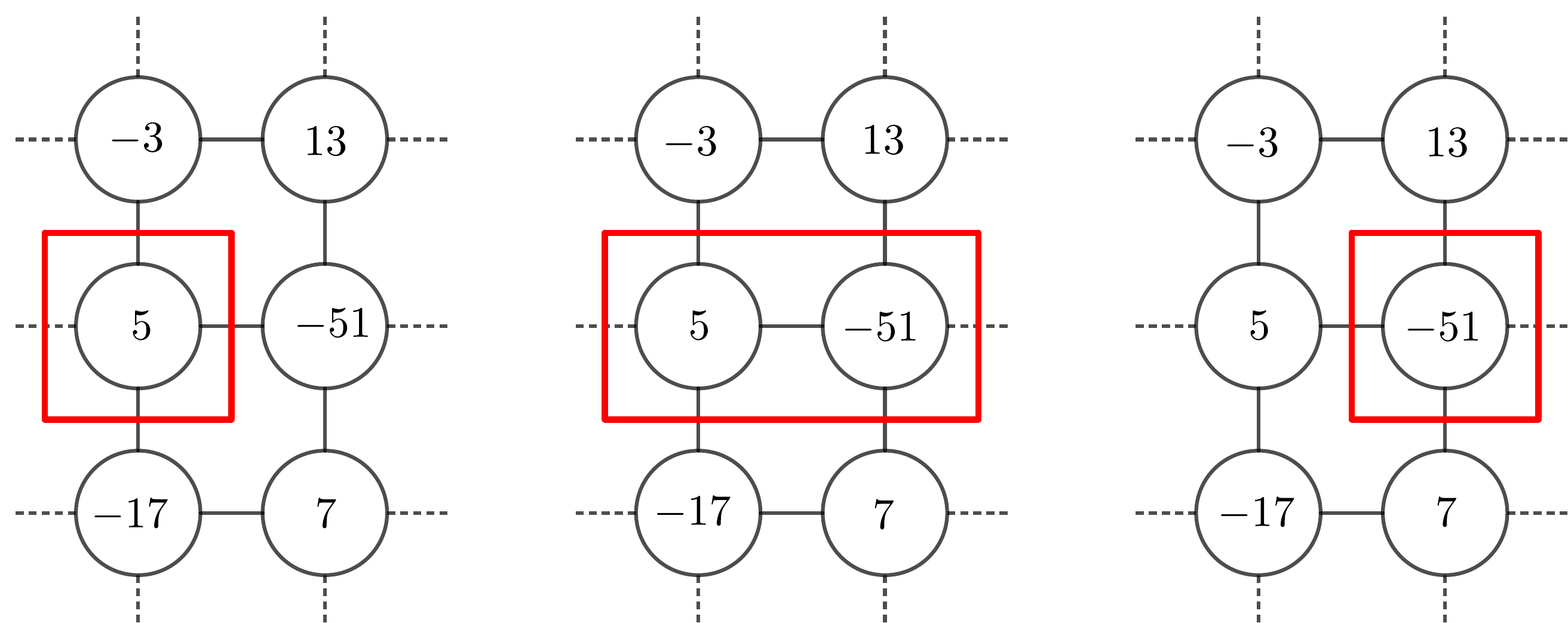}
\end{center}
\caption{Passing from a vertex to an adjacent vertex by elementary moves.}
\label{figure2}
\end{figure}

\medskip \noindent
In order to link two wreaths $(R_1, \varphi_1)$ and $(R_2, \varphi_2)$ by a path in $\mathfrak{W}$, we need to modify the integers at the points on which $\varphi_1,\varphi_2$ differ and to find a sequence of rectangles from $R_1$ to $R_2$ such that a rectangle is obtained from the previous one by an elementary move. Notice that, if we want to modify the integer at some point $p \in \mathbb{Z}^2$, then one of our rectangles must contain $p$ in its interior, and $|\varphi_1(p)- \varphi_2(p)|$ elementary moves will be needed to transform $\varphi_1(p)$ to $\varphi_2(p)$. Therefore, the distance between $(R_1, \varphi_1)$ and $(R_2, \varphi_2)$ in $\mathfrak{W}$ is equal to 
$$TC(R_1, \varphi_1 \Delta \varphi_2,R_2)+ \sum\limits_{p \in \mathbb{Z}^2} | \varphi_1(p)- \varphi_2(p)|,$$
where $\varphi_1 \Delta \varphi_2$ denotes the set of points on which $\varphi_1, \varphi_2$ differ and $TC(R_1,F,R_2)$ the minimal number of rectangles needed to connect $R_1$ to $R_2$ in such a way that any point of $F \subset \mathbb{Z}^2$ belongs to one of these rectangles. It is worth noticing that applying an elementary move to some rectangle $R$ amounts to adding or removing a hyperplane of $R$. With this idea in mind, it can be proved that
$$TC(R_1,F,R_2) = 2 \cdot \# \mathcal{H}(R_1 \cup R_2 \cup F) - \# \mathcal{H}(R_1) - \# \mathcal{H}(R_2),$$
where $\mathcal{H}(S)$ denotes the set of hyperplanes separating at least two vertices of $S$. The idea is essentially the following: if $J$ is a hyperplane separating two vertices of $R_1 \cup R_2 \cup F$, then in our sequence of rectangles from $R_1$ to $R_2$, we will need to add $J$ to one of these rectangles and next to remove it from another one, except if $J$ already crosses $R_1$ (so that we do not need to add it) or if it crosses $R_2$ (so that we do not need to remove it). See \cite[Section 9]{Qm} for more information. Thus, the distance between $(R_1, \varphi_1)$ and $(R_2, \varphi_2)$ in the graph of wreaths $\mathfrak{W}$ is equal to
$$2 \cdot \# \mathcal{H}(R_1 \cup R_2 \cup \varphi_1 \Delta \varphi_2) - \# \mathcal{H}(R_1) - \# \mathcal{H}(R_2)+ \sum\limits_{p \in \mathbb{Z}^2} | \varphi_1(p)- \varphi_2(p)|.$$
In the next sections, we will generalise these ideas to arbitrary median spaces.

\section{Diadem products of median spaces}

\subsection{Preliminaries on median spaces}\label{section:medianspaces}

\noindent
In this section, we give the preliminary material on median spaces which will be needed in the sequel. We refer to \cite{medianviewpoint} and references therein for more information. 

\begin{definition}
Let $X$ be a metric space. Given two points $x,y \in X$, the \emph{interval} between $x$ and $y$ is
$$I(x,y) = \{ z \in X \mid d(x,y)=d(x,z)+d(z,y) \}.$$
Given any three points $x,y,z \in X$, a point in the intersection $I(x,y) \cap I(y,z) \cap I(x,z)$ is a \emph{median point} of $x$, $y$ and $z$. The space $X$ is \emph{median} if any triple of points admits a unique median.
\end{definition}

\noindent
Important examples of median spaces are \emph{median graphs}, since it was proved independently in \cite{Roller, mediangraphs} that they are precisely the one-skeletons of CAT(0) cube complexes. In fact, median spaces can be thought of as a ``non-discrete'' generalisation of these complexes. In particular, the technology of \emph{hyperplanes} can be extended. 

\begin{definition}
Let $X$ be a median space. A subspace $Y \subset X$ is \emph{convex} if $I(x,y) \subset Y$ for every $x,y \in Y$. A \emph{halfspace} of $X$ is a convex subspace whose complement is convex as well. Finally, a \emph{hyperplane} of $X$ is a pair $\{D,D^c\}$ where $D$ is a halfspace.
\end{definition}

\noindent
In a median graph, the distance between any two vertices coincides with the number of hyperplanes separating them. In order to generalise this idea to median spaces, we need to introduce \emph{measured wallspaces}. 

\begin{definition}
Let $X$ be a set. A \emph{wall} $W$ is a partition $\{D,D^c\}$ of $X$ into two non empty subsets; $D$ and $D^c$ are referred to as the \emph{halfspaces} delimited by $W$. Two points $x,y \in X$ are \emph{separated} by a given wall $\{Y,Y^c \}$ if either $x \in Y$ and $y \in Y^c$, or $x \in Y^c$ and $y \in Y$.
\end{definition}

\noindent
The typical examples of walls we have in mind are hyperplanes in median spaces.

\begin{definition}
A \emph{measured wallspace} $(X, \mathcal{W}, \mathcal{B}, \mu)$ is the data of a set $X$, a collection of walls $\mathcal{W}$, a $\sigma$-algebra $\mathcal{B}$ of $\mathcal{W}$ and $\mu$ an associated measure, such that, for every points $x,y \in X$ the collection of walls $\mathcal{W}(x \mid y)$ separating $x$ and $y$ belongs to $\mathcal{B}$ and has finite $\mu$-measure. 
\end{definition}

\noindent
It is proved in \cite{medianviewpoint} that a median space, together with its collection of hyperplanes, can be naturally endowed with a structure of measured wallspace which is compatible with the initial metric. More precisely,

\begin{thm}
Let $(X,d)$ be a median space. There exist a $\sigma$-algebra $\mathcal{B}$ and a measure $\mu$ defined on the set of hyperplanes of $X$ such that, for every points $x,y \in X$, $\mathcal{W}(x \mid y)$ belongs to $\mathcal{B}$ and $\mu~\mathcal{W}(x \mid y)=d(x,y)$. 
\end{thm}

\noindent
Another useful tool in the study of median spaces is that it is possible to define \emph{projections} on some subspaces.

\begin{definition}
Let $X$ be a metric space and $Y \subset X$ a subspace. Given two points $x \in X$ and $p \in Y$, $p$ is a \emph{gate} for $x$ in $Y$ if $p \in I(x,y)$ for every $y \in Y$. If every point of $X$ admits a gate in $Y$, we say that $Y$ is \emph{gated}.  
\end{definition}

\noindent
Clearly, if it exists, a gate of a point $x$ is the unique point of the subspace which minimises the distance to $x$. In particular, for any gated subspace $Y$, it allows to define the \emph{projection} of any point $x \in X$ onto $Y$ as the unique gate of $x$ in $Y$.

\begin{lemma}\label{lem:proj1}
Let $X$ be a median space, $C \subset X$ a gated subspace and $x \in X$ a point. Any hyperplane separating $x$ from its projection onto $C$ separates $x$ from $C$.
\end{lemma}

\begin{proof}
Let $x' \in C$ denote the projection of $x$ onto $C$, and let $\{D,D^c\}$ be a hyperplane separating $x$ and $x'$, say $x' \in D$ and $x \in D^c$. For any point $z \in D^c$, necessarily $I(x,z) \subset D^c$ by convexity. On the other hand, if $z \in C$, then $I(x,z) \cap D \neq \emptyset$ since $x' \in I(x,z)$. Therefore, $z \in D$. This proves that $C \subset D$, so that $\{D,D^c \}$ separates $x$ from $C$.
\end{proof}

\noindent
For instance, it is proved in \cite{medianviewpoint} that closed convex subspaces in complete median spaces are gated. In this paper, we are interested in the class of \emph{finitely generated} convex subspaces.

\begin{definition}
In a median space $X$, a convex subspace is \emph{finitely generated} if it is the convex hull of finitely many points. We denote by $\mathcal{F}(X)$ the collection of all the non empty finitely generated convex subspaces of $X$. 
\end{definition}

\noindent
Our main lemma about finitely generated convex subspaces is the following:

\begin{lemma}\label{lem:fgmainlem}
Let $X$ be a median space and $C_1,C_2 \in \mathcal{F}(X)$ two subspaces. There exist two points $x_1 \in C_1$ and $x_2 \in C_2$ such that 
$$\mathcal{W}(x_1 \mid x_2) = \mathcal{W}(C_1 \mid C_2) \ \text{and} \ d(x_1,x_2)=d(C_1,C_2).$$ 
Moreover, $x_1$ is a gate of $x_2$ in $C_1$ and similarly $x_2$ is a gate of $x_1$ in $C_2$. 
\end{lemma}

\begin{proof}
For any subset $F \subset X$, define $M(F)= \{ m(x,y,z) \mid x,y,z \in F \}$, and by induction 
$$\left\{ \begin{array}{l} M^0(F)=F \\ M^{n+1}(F)=M(M^n(F)) \ \text{for every $n \geq 0$} \end{array} \right. .$$ 
By construction, $\bigcup\limits_{n \geq 0} M^n(F)$ coincides with the \emph{median hull} of $F$, ie., the smallest subset of $X$ containing $F$ which is stable under the median operation. Moreover, because the median hull of a finite set turns out to be finite according to \cite[Lemma 6.20]{vantheory}, there exists some $N \geq 0$ such that $M^n(F)= \bigcup\limits_{n \geq 0} M^n(F)$ for every $n \geq N$. 

\medskip \noindent
Let $F_1,F_2 \subset X$ be two finite subsets such that $C_1$ and $C_2$ are the convex hulls of $F_1$ and $F_2$ respectively. Let $F$ denote the median hull of $F_1 \cup F_2$; according to our previous observation, $F$ is finite. We claim that $F \subset C_1 \cup C_2$. It is clear that $M^0(F_1 \cup F_2) \subset C_1 \cup C_2$; and if $M^n(F_1 \cup F_2) \subset C_1 \cup C_2$ for some $n \geq 0$, then any point $p \in M^{n+1}(F_1 \cup F_2)$ can be written as $p=m(x,y,z)$ for some $x,y,z \in C_1 \cup C_2$, say with $x,y \in C_1$, so that $p \in I(x,y) \subset C_1$. Thus, it follows by induction that $M^n(F_1 \cup F_1) \subset C_1 \cup C_2$ for every $n \geq 0$, hence $F \subset C_1 \cup C_2$. We have proved more generally that

\begin{fact}\label{fact:unionfg}
If $C_1$ and $C_2$ are the convex hulls of two subsets $F_1$ and $F_2$ respectively, then the median hull of $F_1 \cup F_2$ is included into $C_1 \cup C_2$. 
\end{fact}

\noindent
Now, fix two points $x_1 \in F \cap C_1$ and $x_2 \in F \cap C_2$ satisfying
$$d(x_1,x_2)= \min \left\{ d(x,y) \mid x \in F \cap C_1, y \in F \cap C_2) \right\}.$$
Let $z \in F \cap C_1$ be a point. Because the median point $m$ of $x_1$, $z$ and $x_2$ necessarily belongs to $F \cap C_1$ and that $d(x_1,x_2)=d(x_1,m)+d(m,x_2)$, we deduce that $m=x_1$, so that $x_1 \in I(z,x_2)$. As a consequence, any hyperplane separating $x_1$ and $x_2$ must separate $z$ and $x_2$. Indeed, if $\{ D,D^c \}$ is such a hyperplane, say with $x_2 \in D$ and $x_1 \in D^c$, and if $z$ belongs to $D$, then it follows that $x_1 \in I(z,x_2) \subset D$ by convexity of $D$, which is absurd. Thus, we have proved that any hyperplane separating $x_1$ and $x_2$ separates $F \cap C_1$ and $x_2$. By symmetry, our argument also implies that any hyperplane separating $x_1$ and $x_2$ separates $x_1$ and $F \cap C_2$. Therefore, $\mathcal{W}(x_1 \mid x_2) \subset \mathcal{W}(F_1 \mid F_2)$. The reverse inclusion being clear, it follows that $\mathcal{W}(x_1 \mid x_2) = \mathcal{W}(F_1 \mid F_2)$. From the inequalities
$$d(C_1,C_2) \leq d(x_1,x_2)= \mu~\mathcal{W}(x_1 \mid x_2) = \mu~\mathcal{W}(F_1 \mid F_2) \leq d(C_1,C_2),$$
we conclude that $d(x_1,x_2)=d(C_1,C_2)$. 

\medskip \noindent
Now, we want to prove that $x_2$ is a gate of $x_1$ in $C_2$. So fix a point $w \in C_2$. If $J$ is a hyperplane separating $x_2$ and $w$, then $J$ does not separate $x_1$ and $x_2$, because we know that the hyperplanes separating $x_1$ and $x_2$ are precisely the hyperplanes separating $C_1$ and $C_2$, which do not intersect $C_2$ in particular. Equivalently, $\mathcal{W}(x_2 \mid w) \cap \mathcal{W}(x_1,x_2)= \emptyset$. As a consequence, $\mathcal{W}(x_2 \mid w) \subset \mathcal{W}(x_1 \mid w)$. Because any hyperplane separarating $x_1$ and $x_2$ must separate $C_1$ and $C_2$, and a fortiori $x_1$ and $w$, it follows that
$$\mathcal{W}(x_1 \mid w) = \mathcal{W}(x_1 \mid x_2) \sqcup \mathcal{W}(x_2 \mid w),$$
hence $d(x_1,w)=d(x_1,x_2)+d(x_2,w)$. Thus, we have proved that $x_2$ is a gate of $x_1$ in $C_2$. A symmetric argument proves that $x_1$ is a gate of $x_2$ in $C_1$. 
\end{proof}

\noindent
As a consequence of Lemma \ref{lem:fgmainlem}, it follows that finitely generated convex subspaces are gated, so that it will be possible to project points on such subspaces.

\begin{cor}
In a median space, any finitely generated convex subspace is gated.
\end{cor}

\begin{proof}
Let $X$ be a median space, $C \in \mathcal{F}(X)$ some subspace and $x \in X$ some point. Applying Lemma \ref{lem:fgmainlem} to $\{x\}$ and $C$ provides the conclusion.
\end{proof}

\noindent
It is known that, in median spaces, any two disjoint convex subspaces are separated by at least one hyperplane. Another consequence of Lemma \ref{lem:fgmainlem} is that, if these two subspaces are moreover finitely generated, then the collection of the hyperplanes separating them is measurable and has positive measure.

\begin{cor}\label{cor:proj2}
Let $X$ be a median graph and $C_1, C_2 \in \mathcal{F}(X)$ two subspaces. If $C_1$ and $C_2$ are disjoint, then $\mu~\mathcal{W}(C_1 \mid C_2)>0$. 
\end{cor}

\begin{proof}
Let $x_1 \in C_1$ and $x_2 \in C_2$ be the two points given by Lemma \ref{lem:fgmainlem}. Notice that, because $C_1$ and $C_2$ are disjoint, necessarily $x_1 \neq x_2$. We have
$$\mu~\mathcal{W}(C_1 \mid C_2)=\mu~\mathcal{W}(x_1 \mid x_2)=d(x_1,x_2)>0,$$
which proves our corollary.
\end{proof}

\noindent
Finally, we conclude this section by noticing that being finitely generated is stable under intersection.

\begin{lemma}\label{lem:fginter}
Let $X$ be a median space and $C_1,C_2 \in \mathcal{F}(X)$ two subspaces. The intersection $C_1 \cap C_2$ is finitely generated. 
\end{lemma}

\begin{proof}
Let $F_1,F_2 \subset X$ be two finite subsets such that $C_1$ and $C_2$ are the convex hulls of $F_1$ and $F_2$ respectively. According to Fact \ref{fact:unionfg}, the median hull $F$ of $F_1 \cup F_2$ is included into $C_1 \cup C_2$. Let $Q$ denote the convex hull of $F \cap C_1 \cap C_2$. Notice that, because the convex hull of $F$ contains $C_1 \cup C_2$, necessarily $C_1 \cap C_2 \subset Q$. The reverse inclusion being clear, it follows that $Q= C_1 \cap C_2$. Thus, $C_1 \cap C_2$ is the convex hull of $F$, which is finite according to \cite[Lemma 6.20]{vantheory}. A fortiori, $C_1 \cap C_2$ is finitely generated. 
\end{proof}

\subsection{The space of finitely generated convex subspaces}\label{section:fgconvex}

\noindent
Recall that, given a median space, a convex subspace is \emph{finitely generated} if it is the convex hull of finitely many points of $X$. Notice that, if $C$ is such a subspace, then the set $\mathcal{H}(C)$ of the hyperplanes intersecting $C$ is measurable and has finite measure. Indeed, if $C$ is the convex hull of some finite set $\{ x_1, \ldots, x_n\}$, then $\displaystyle \mathcal{H}(C)= \bigcup\limits_{1 \leq i < j \leq n} \mathcal{W}(x_i \mid x_j)$ and $\displaystyle\mu(\mathcal{H}(C)) \leq \sum\limits_{1 \leq i < j \leq n} d(x_i,x_j)$. The goal of this section is to exploit this observation in order to define a median metric on the set of finitely generated convex subspaces of a given median space.

\medskip \noindent
In the sequel, we will use the following notation. Fix a median space $X$. For any subset $F \subset X$, we denote by $\mathcal{H}(F)$ the set of the hyperplanes separating two points of $F$; alternatively, this is also the set of the hyperplanes intersecting the convex hull of $F$. If $A_1, \ldots, A_n \subset X$ are subsets such that the convex hull of $A_1 \cup \cdots \cup A_n$ is finitely generated, we denote by $\mu(A_1 \cup \cdots \cup A _n)$ the measure of $\mathcal{H}(A_1 \cup \cdots \cup A_n)$. 

\begin{definition}
Given a median space $X$, we denote by $\mathcal{F}(X)$ the set of non empty finitely generated convex subspaces of $X$, which we equip with the map $d : \mathcal{F}(X) \times \mathcal{F}(X) \to \mathbb{R}_+$ defined by
$$d : (C_1,C_2) \mapsto 2 \cdot \mu(C_1 \cup C_2)- \mu(C_1)- \mu(C_2).$$
\end{definition}

\noindent
The rest of the section is dedicated to the proof of the following statement.

\begin{prop}\label{prop:Fmedian}
$(\mathcal{F}(X),d)$ is a median space.
\end{prop}

\noindent
The first thing to verify is that $d$ defines indeed a distance on $\mathcal{F}(X)$. 

\begin{lemma}\label{lem:dmetric}
$(\mathcal{F}(X),d)$ is a metric space. 
\end{lemma}

\begin{proof}
The map $d$ is clearly symmetric. Now, let $C_1,C_2 \in \mathcal{F}(X)$ be two distinct convex subspaces. Say that there exists some $x \in C_1 \backslash C_2$. Notice that
$$\begin{array}{lcl} d(C_1,C_2) & = & 2 \cdot \mu(C_1 \cup C_2)- \mu(C_1)- \mu(C_2) \\ \\ & = & 2 \cdot ( \mu( \mathcal{H}(C_1) \backslash \mathcal{H}(C_2)) + \mu( \mathcal{H}(C_2) \backslash \mathcal{H}(C_1)) + \mu( \mathcal{H}(C_1) \cap \mathcal{H}(C_2)) \\ & & + \mu \mathcal{W}(C_1 \mid C_2) ) -\mu(C_1)- \mu(C_2) \\ \\ & = & \mu( \mathcal{H}(C_1) \backslash \mathcal{H}(C_2)) + \mu( \mathcal{H}(C_2) \backslash \mathcal{H}(C_1)) + \mu \mathcal{W}(C_1 \mid C_2) \end{array}$$
On the other hand, if $x'$ denotes the projection of $x$ onto $C_2$, then any hyperplane separating $x$ and $x'$ must separate $x$ and $C_2$ according to Lemma \ref{lem:proj1}, so that
$$\mathcal{W}(x \mid x') \subset \left( \mathcal{H}(C_1) \backslash \mathcal{H}(C_2) \right) \cup \mathcal{W}(C_1 \mid C_2).$$
Therefore, we deduce that
$$d(C_1,C_2) \geq \mu \mathcal{W}(x \mid x') = d_X(x,x')$$
which is positive because $x$ does not belong to $C_2$. Thus, we have proved that $d$ is positive-definite.

\medskip \noindent
Next, we want to prove the triangle inequality. So let $C_1,C_2,C_3 \in \mathcal{F}(X)$ be three convex subspaces. First of all, notice that

\begin{claim}\label{claim:11triangleinequ}
The following inequality holds:
$$\mathds{1}_{\mathcal{H}(C_1 \cup C_3)} \leq \mathds{1}_{\mathcal{H}(C_1 \cup C_2)} + \mathds{1}_{\mathcal{H}(C_2 \cup C_3)} - \mathds{1}_{\mathcal{H}(C_2)}$$
\end{claim}

\noindent
Indeed, for every hyperplane $J$ of $X$, if we denote respectively by $L$ and $R$ the left-hand-side and the right-hand-side of the previous inequality, then
\begin{itemize}
	\item if $J$ intersects either both $C_1$ and $C_2$, or both $C_2$ and $C_3$, then $L(J)=1=R(J)$; 
	\item if $J$ intersects either $C_1$ but not $C_2$, or $C_3$ but not $C_2$, then $L(J)=1$ and $R(J) \geq 1$;
	\item if $J$ intersects $C_2$ but not $C_1$ nor $C_3$, then $L(J) \leq 1$ and $R(J)=1$;
	\item if $J$ delimits a halfspace containing $C_1,C_2,C_3$, then $L(J)=0=R(J)$;
	\item if $J$ separates $C_2$ and $C_1 \cup C_3$, then $L(J)=0$ and $R(J)=2$;
	\item if $J$ separates either $C_1$ and $C_2 \cup C_3$, or $C_3$ and $C_1 \cup C_2$, then $L(J)=1=R(J)$.
\end{itemize}
This proves our claim. By integrating this inequality, we deduce that
$$\mu(C_1 \cup C_3) \leq \mu(C_1 \cup C_2)+ \mu(C_2 \cup C_3)- \mu (C_2).$$
As a consequence,
$$\begin{array}{lcl} d(C_1,C_2)+d(C_2,C_3) & = & 2 \left( \mu(C_1 \cup C_2)+ \mu(C_2 \cup C_3)- \mu (C_2) \right) - \mu(C_1)- \mu(C_3) \\ \\ & \geq & \mu(C_1 \cup C_3)- \mu(C_1)- \mu(C_3)= d(C_1,C_3) \end{array}$$
which proves the triangle inequality. 
\end{proof}

\noindent
The next step towards the proof of Proposition \ref{prop:Fmedian} is to understand the intervals in our metric space.

\begin{lemma}\label{lem:intervalFX}
Let $X$ be a median space and $C,C_1,C_2 \in \mathcal{F}(X)$ three convex subspaces. The point $C$ belongs to the interval between $C_1$ and $C_2$ in $\mathcal{F}(X)$ if and only if the following three conditions are satisfied:
\begin{itemize}
	\item[(i)] $C$ is included into the convex hull of $C_1 \cup C_2$;
	\item[(ii)] any hyperplane intersecting both $C_1$ and $C_2$ must intersect $C$;
	\item[(iii)] no hyperplane intersecting $C_1$ separates $C$ and $C_2$, and similarly no hyperplane intersecting $C_2$ separates $C$ and $C_1$.
\end{itemize}
\end{lemma}

\begin{proof}
Because
$$d(C_1,C)+d(C,C_2) = 2 \cdot (\mu(C_1 \cup C) + \mu(C \cup C_2) -\mu(C)) - \mu(C_1)- \mu(C_2)$$
and
$$d(C_1,C_2)= \mu(C_1 \cup C_2)- \mu(C_1)- \mu(C_2),$$
it follows that $C$ belongs to $I(C_1,C_2)$ if and only if the equality
\begin{equation}\label{eq:mu}
\mu(C_1 \cup C) + \mu(C \cup C_2) -\mu(C)= \mu(C_1 \cup C_2) 
\end{equation}
holds. Suppose that the three conditions of our statement hold. We want to prove that
\begin{equation}\label{eq:11}
\mathds{1}_{\mathcal{H}(C_1 \cup C_2)} = \mathds{1}_{\mathcal{H}(C_1 \cup C)} + \mathds{1}_{\mathcal{H}(C_2 \cup C)} - \mathds{1}_{\mathcal{H}(C)}
\end{equation}
so that the previous equality will follow by integration. For every hyperplane $J$ of $X$, if we denote respectively by $L$ and $R$ the left-hand-side and the right-hand-side of our equality above, then
\begin{itemize}
	\item if $J$ intersects either both $C_1$ and $C$, or both $C_2$ and $C$, then $L(J)=1=R(J)$; 
	\item if $J$ intersects $C_1$ but not $C$, then $J$ cannot intersect $C_2$ by condition $(ii)$ and it cannot separate $C_2$ and $C$ by condition $(iii)$, hence $L(J)=1 =R(J)$; if $J$ intersects $C_2$ but not $C$, the situation is symmetric;
	\item if $J$ intersects $C$ but not $C_1$ nor $C_2$, then $J$ must separate $C_1$ and $C_2$ by condition $(i)$, so that $L(J) = 1 = R(J)$;
	\item if $J$ delimits a halfspace containing $C_1,C_2,C$, then $L(J)=0=R(J)$;
	\item $J$ cannot separate $C$ from $C_1 \cup C_2$ by condition $(i)$;
	\item if $J$ separates either $C_1$ and $C \cup C_2$, or $C_2$ and $C_1 \cup C$, then $L(J)=1=R(J)$.
\end{itemize}
Thus, we have proved that, if $C$ satisfies the conditions $(i)$, $(ii)$ and $(iii)$, then it belongs to $I(C_1,C_2)$.

\medskip \noindent
Conversely, if we denote respectively by $L$ and $R$ the left-hand-side and the right-hand-side of the equality \ref{eq:11}, we claim that, if $C$ does not satisfy one of the conditions $(i)$, $(ii)$ or $(iii)$, then the inequality $L < R$ holds on a set of positive measure. Because we already know from Claim \ref{claim:11triangleinequ} that the inequality $L \leq R$ holds everywhere, it follows by integrating this inequality that the equality \ref{eq:mu} cannot hold, so that $C$ cannot belong to the interval $I(C_1,C_2)$. 
\begin{itemize}
	\item If $C$ does not satisfy the condition $(i)$, there exists a point $x \in C$ which does not belong to the convex hull of $C_1 \cup C_2$. Let $x'$ denote the projection of $x$ onto this convex hull. According to Lemma \ref{lem:proj1}, any hyperplane separating $x$ from $x'$ must separate $x$ from the convex hull of $C_1 \cup C_2$, so that $L(J)=0 < 1 \leq R(J)$ for every $J \in \mathcal{W}(x \mid x')$. On the other hand, $\mu \mathcal{W}( x \mid x' ) = d(x,x')$ is positive.
	\item If $C$ does not satisfy either the condition $(ii)$ or the condition $(iii)$, there exists a halfspace $D$ intersecting both $C_1$ and $C_2$ but which is disjoint from $C$. Let $F_1,F_2 \subset X$ be two finite subsets such that $C_1$ and $C_2$ are the convex hulls of $F_1$ and $F_2$ respectively. Denote by $A$ the convex hull of $(F_1 \cap D) \cup (F_2 \cap D)$, and by $B$ the convex hull of $(F_1 \cap D^c) \cup (F_2 \cap D^c) \cup C$. Notice that $A$ and $B$ are non empty two finitely generated convex subspaces separated by the hyperplane $\{D,D^c\}$. Moreover, $L(J) \leq 1 < 2= R(J)$ for every $J \in \mathcal{W}(A \mid B)$. On the other hand, because $A$ and $B$ are disjoint, we deduce from Corollary \ref{cor:proj2} that $\mathcal{W}(A \mid B)$ has positive measure. 
\end{itemize}
This concludes the proof of our lemma.
\end{proof}

\begin{proof}[Proof of Proposition \ref{prop:Fmedian}.]
Let $C_1,C_2,C_3 \in \mathcal{F}(X)$ be three convex subspaces. Let $M$ denote the intersection of the convex hulls of $C_1 \cup C_2$, $C_2 \cup C_3$ and $C_1 \cup C_3$. Notice that $M$ is finitely generated according to Lemma \ref{lem:fginter}, and is non empty because $m(x_1,x_2,x_3) \in M$ for every $x_1 \in C_1$, $x_2 \in C_2$ and $x_3 \in C_3$. According to Lemma \ref{lem:intervalFX},
$$I(C_1,C_2) \cap I(C_2,C_3) \cap I(C_1,C_3) \subset \{ C \in \mathcal{F}(X) \mid C \subset M \}.$$
Let $C \in \mathcal{F}(X)$ be a convex subspace satisfying $C \subsetneq M$. Fix a point $x \in M \backslash C$, let $x'$ denote its projection onto $C$ and let $J$ be a hyperplane separating $x$ and $x'$. Notice that, according to Lemma \ref{lem:proj1}, $J$ separates $x$ and $x'$. Moreover, two subcomplexes among $C_1,C_2,C_3$ cannot be both included into some halfspace $D$ delimited by $J$ since otherwise the convex hull of the union of these two subcomplexes, and a fortiori $M$, would be included into $D$, which is impossible because $J$ separates two points of $M$, namely $x$ and $x'$. Therefore, $J$ intersects at least one subcomplex among $C_1,C_2,C_3$, say $C_1$, and either separates $C_2$ and $C_3$ or intersects at least one of $C_2$ and $C_3$. In the former case, if $C$ belongs to the same halfspace delimited by $J$ as $C_2$, say, then we deduce from Lemma \ref{lem:intervalFX} that $C$ does not belong to $I(C_1,C_3)$; in the latter case, if $J$ intersects both $C_1$ and $C_2$, say, then we also deduce from Lemma \ref{lem:intervalFX} that $C$ does not belong to $I(C_1,C_2)$.

\medskip \noindent
Thus, we have proved that $M$ is the only candidate for a median point of $C_1,C_2,C_3$. We claim that $M$ is such a median point. 

\medskip \noindent
Let $J$ be a hyperplane intersecting both $C_1$ and $C_2$. So there exist points $x_1,y_1 \in C_2$ and $x_2,y_2 \in C_2$ such that $J$ separates $x_1$ and $y_1$, and $x_2$ and $y_2$; say that $x_1$ and $x_2$ belong to the same halfspace delimited by $J$. Fix an arbitrary point $z \in C_3$. Since halfspaces are convex, it follows that $m(x_1,x_2,z)$ belongs to the halfspace delimited by $J$ containing $x_1$ and $x_2$, and that $m(y_1,y_2,z)$ belongs to the halfspace delimited by $J$ containing $y_1$ and $y_2$, so $J$ separates the two points $m(x_1,x_2,z)$ and $m(y_1,y_2,z)$ of $M$. A fortiori, $J$ intersects $M$. Now, suppose by contradiction that there exists a hyperplane $J$ intersecting $C_1$ which separates $M$ and $C_2$. As a consequence of our previous observation, $J$ cannot intersect $C_3$. Moreover, $C_3$ cannot be included into the halfspace delimited by $J$ which contains $C_2$, because otherwise the convex hull of $C_2 \cup C_3$ and $M$ would be separated by $J$, which impossible by the definition of $M$. Therefore, $J$ separates $C_2$ and $C_3$. Fix two arbitrary points $x_2 \in C_2$ and $x_3 \in C_3$, and fix a point $x_1 \in C_1$ which belongs to the same halfspace delimited by $J$ as $x_2$. Since halfspaces are convex, it follows that the point $m(x_1,x_2,x_3)$ of $M$ belongs to the same halfspace delimited by $J$ as $C_2$, which contradicts the assumption that $J$ separates $C_2$ and $C$. Therefore, no hyperplane intersecting $C_1$ separates $C$ and $C_2$; and similarly, no hyperplane intersecting $C_2$ separates $C$ and $C_3$. 

\medskip \noindent
Thanks to Lemma \ref{lem:intervalFX}, we conclude that $M$ belongs to the interval $I(C_1,C_2)$. By symmetry, we deduce that $M$ also belongs to the intervals $I(C_1,C_3)$ and $I(C_2,C_3)$, so that $M \in I(C_1,C_2) \cap I(C_2,C_3) \cap I(C_1,C_3)$, ie., $M$ is a median point of $C_1,C_2,C_3$. 
\end{proof}

\subsection{The space of wreaths}\label{section:spacewreath}

\noindent
We are now ready to define diadem products of median spaces and to study their geometry.

\begin{definition}
Let $X,Y$ be two median spaces and $1 \in X$ a basepoint. The \emph{diadem product} $(X,1)\circledast Y$ is the set of \emph{wreaths} $(C,\varphi)$, where $Y \in \mathcal{F}(Y)$ and where $\varphi : Y \to X$ satisfies $\varphi(y)=1$ for all but finitely many $y \in Y$ (written $\varphi \in X^{(Y)}$ in the sequel), endowed with the metric $\delta$ defined as
$$((C_1,\varphi_1),(C_2,\varphi_2)) \mapsto 2 \cdot \mu(C_1 \cup C_2 \cup \varphi_1 \Delta \varphi_2)- \mu(C_1)- \mu(C_2)+ \sum\limits_{y \in Y} d(\varphi_1(y), \varphi_2(y)).$$
\end{definition}

\noindent
The fact that $\delta$ is indeed a metric will be justified later; see Corollary \ref{cor:Isadist}. The main result of this section is the following:

\begin{thm}\label{thm:Wdeltamedian}
A diadem product of two median spaces is a median space.
\end{thm}

\noindent
From now on, we fix two median spaces $X,Y$ and a basepoint $1 \in X$, and for short we denote by $\mathfrak{W}$ the diadem product $(X,1) \circledast Y$. Before proving the theorem, we need to introduce some preliminary material. 

\begin{definition}
A \emph{leaf} of $\mathfrak{W}$ is a subspace $\mathfrak{W}(\varphi):= \{ (C, \varphi) \mid C \in \mathcal{F}(Y) \},$ the map $\varphi \in X^{(Y)}$ being fixed.
\end{definition}

\noindent
Clearly, the map $C \mapsto (C,\varphi)$ defines an isometry $\mathcal{F}(Y) \to \mathfrak{W}(\varphi)$, so that we already understand the geometry of the leaves of $\mathfrak{W}$ thanks to the previous section. Fixing a leaf $\mathfrak{W}(\varphi)$, we define a \emph{projection} 
$$p_{\varphi} : \left\{ \begin{array}{ccc} \mathfrak{W} & \to & \mathfrak{W}(\varphi) \\ \\ (C,\psi) & \mapsto & \left( \overline{C \cup \psi \Delta \varphi}, \varphi \right) \end{array} \right. ,$$
where $\overline{\cdot}$ denotes the convex hull. As a consequence of our first preliminary lemma below, this map is a ``true'' projection, in the sense that $p_{\varphi}(x)$ is the unique point of the leaf $\mathfrak{W}(\varphi)$ minimising the distance to a given point $x$. 

\begin{lemma}\label{lem:pvarphi}
For every $\varphi \in X^{(Y)}$, every $x \in \mathfrak{W}$ and every $y \in \mathfrak{W}(\varphi)$, the following equality holds
$$\delta(x,y)= \delta(x, p_{\varphi}(x))+ \delta(p_{\varphi}(x),y).$$
\end{lemma}

\begin{proof}
If $x=(C,\psi)$ and $y= (Q, \varphi)$, then the sum $ \delta(x,p_{\varphi}(x))+ \delta(p_{\varphi}(x),y)$ simplifies as
$$2 \cdot \mu (C \cup Q \cup \varphi \Delta \psi) - \mu(C)- \mu(Q) + \sum\limits_{y \in Y} d(\varphi(y),\psi(y)),$$
which is precisely $\delta(x,y)$. 
\end{proof}

\noindent
Although this lemma is completely elementary, it has important consequences, and it will turn out to be fundamental in the proof of Theorem \ref{thm:Wdeltamedian}. For instance, we are able to show that $\delta$ defines a distance on $\mathfrak{W}$. 

\begin{cor}\label{cor:Isadist}
$(\mathfrak{W}, \delta)$ is a metric space. 
\end{cor}

\begin{proof}
First of all, notice that the map $\delta$ is clearly symmetric.

\medskip \noindent
Next, if two wreaths $(C_1, \varphi_1),(C_2, \varphi_2) \in \mathfrak{W}$ satisfy $\delta((C_1,\varphi_1),(C_2, \varphi_2))=0$, then necessarily $\sum\limits_{y \in Y} d(\varphi_1(y), \varphi_2(y))=0$ for every $y \in Y$. This implies that $\varphi_1=\varphi_2$, ie., our two wreaths belong to a common leaf $\mathfrak{W}(\varphi)$. On the other hand, the restriction of $\delta$ to this leaf, namely $((Q_1,\varphi),(Q_2, \varphi)) \mapsto d(Q_1,Q_2)$, is a distance according to Lemma \ref{lem:dmetric}. Consequently, $C_1$ must be equal to $C_2$, so that $(C_1, \varphi_1)= (C_2, \varphi_2)$. We have proved that $\delta$ is positive-definite.

\medskip \noindent
Finally, for any three wreaths $x=(C_1, \varphi_1)$, $y=(C_2, \varphi_2)$ and $z=(C, \varphi)$, we deduce from Lemma \ref{lem:pvarphi} that
$$\delta(x,z)+\delta(z,y) = \delta(x,p_{\varphi}(x))+ \delta(p_{\varphi}(x),z)+ \delta(z, p_{\varphi}(y)) + \delta(p_{\varphi}(y),y).$$
On the other hand, since we know from Lemma \ref{lem:dmetric} that the restriction of $\delta$ to the leaf $\mathfrak{W}(\varphi)$, is a distance, it follows that $\delta(p_{\varphi}(x),z)+ \delta(z, p_{\varphi}(y)) \geq \delta(p_{\varphi}(x),p_{\varphi}(y))$, hence
$$\delta(x,z)+\delta(z,y) \geq \delta(x,p_{\varphi}(x))+  \delta(p_{\varphi}(x),p_{\varphi}(y)) + \delta(p_{\varphi}(y),y).$$
Notice that the sum in the right-hand-side of this inequality simplifies as
$$2 \cdot \mu ( C_1 \cup C_2 \cup \varphi_1 \Delta \varphi \cup \varphi \Delta \varphi_2) - \mu(C_1)- \mu(C_2).$$
But if $y \in Y$ is a point on which $\varphi_1$ and $\varphi_2$ differ, necessarily either $\varphi(y)$ and $\varphi_1(y)$ or $\varphi(y)$ and $\varphi_2(y)$ will differ as well, ie., $\varphi_1 \Delta \varphi_2 \subset \varphi_1 \Delta \varphi \cup \varphi \Delta \varphi_2$. Therefore,
$$\delta(x,z)+ \delta(z,y) \geq 2 \cdot \mu(C_1 \cup C_2 \cup \varphi_1 \Delta \varphi_2) - \mu(C_1)- \mu(C_2) = \delta(x,y).$$
Thus, $\delta$ satisfies the triangle inequality.
\end{proof}

\noindent
Another consequence of Lemma \ref{lem:pvarphi} is that leaves are convex. 

\begin{cor}\label{cor:leafconvex}
A leaf in $\mathfrak{W}$ is convex.
\end{cor}

\begin{proof}
Let $\varphi \in X^{(Y)}$ be a map, $x,y \in \mathfrak{W}(\varphi)$ two points, and $z \in I(x,y)$ a third point. As a consequence of Lemma \ref{lem:pvarphi}, 
$$\delta(x,y)=\delta(x,z)+ \delta(z,y) = \delta(x,p_{\varphi}(z))+ \delta(p_{\varphi}(z),y)+2 \delta(z, p_{\varphi}(z)).$$
On the other hand, we deduce from the triangle inequality that
$$\delta(x,y) \leq \delta(x,p_{\varphi}(z))+ \delta(p_{\varphi}(z),y).$$
Therefore, $\delta(z,p_{\varphi}(z))=0$, which means that $z$ belongs to the leaf $\mathfrak{W}(\varphi)$. 
\end{proof}

\noindent
Our second (and last) preliminary lemma studies when intervals and leaves intersect. 

\begin{lemma}\label{lem:intervalleaf}
Let $\varphi \in X^{(Y)}$ be a map, and $(C_1, \varphi_1),(C_2, \varphi_2) \in \mathfrak{W}$ two wreaths. The leaf $\mathfrak{W}(\varphi)$ intersects the interval between $(C_1, \varphi_1)$ and $(C_2, \varphi_2)$ if and only if $\varphi(y)$ belongs to $I(\varphi_1(y),\varphi_2(y))$ for every $y \in Y$. 
\end{lemma}

\begin{proof}
For convenience, set $x=(C_1,\varphi_1)$ and $y= (C_2, \varphi_2)$. The interval $I(x,y)$ intersects the leaf $\mathfrak{W}(\varphi)$ if and only if there exists some $z \in \mathfrak{W}(\varphi)$ satisfying $\delta(x,y)=\delta(x,z)+ \delta(z,y)$. This equality is equivalent to
$$\delta(x,y)= \delta(x,p_{\varphi}(x))+ \delta(p_{\varphi}(x),z)+ \delta(z,p_{\varphi}(y))+ \delta(p_{\varphi}(y),y).$$
On the other hand, we know from the triangle inequality that 
$$\delta(x,y) \leq \delta(x,p_{\varphi}(x))+ \delta(p_{\varphi}(x),p_{\varphi}(y))+ \delta(p_{\varphi}(y),y),$$
hence $\delta(p_{\varphi}(x),z)+ \delta(z,p_{\varphi}(y))=\delta(p_{\varphi}(x),p_{\varphi}(y))$. It follows that

\begin{fact}\label{fact:intervalleaf}
The interval $I(x,y)$ intersects the leaf $\mathfrak{W}(\varphi)$ if and only if
$$\delta(x,y)= \delta(x,p_{\varphi}(x))+ \delta(p_{\varphi}(x),p_{\varphi}(y))+ \delta(p_{\varphi}(y),y).$$
\end{fact}

\noindent
This equality simplifies as
\begin{align}\label{eq:intervalleaf}
2 \cdot \mu( C_1 \cup C_2 \cup \varphi_1 \Delta \varphi_2) &+ \sum\limits_{y \in Y} d(\varphi_1(y),\varphi_1(y)) = 2 \cdot \mu( C_1 \cup C_2 \cup \varphi_1 \Delta \varphi \cup \varphi \Delta \varphi_2) \nonumber \\ & \qquad {} + \sum\limits_{y \in Y} \left( d(\varphi_1(y), \varphi(y))+d(\varphi(y), \varphi_1(y)) \right)
\end{align}
Suppose that $I(x,y)$ intersects $\mathfrak{W}(\varphi)$, so that the previous equality holds. From the triangle inequality, it follows that
$$\mu(C_1 \cup C_2 \cup \varphi_1 \Delta \varphi \cup \varphi \Delta \varphi_2) \leq \mu(C_1 \cup C_2 \cup \varphi_1 \Delta \varphi_2).$$
On the other hand, $\varphi_1 \Delta \varphi_2 \subset \varphi_1 \Delta \varphi \cup \varphi \Delta \varphi_2$. Indeed, if $y \in Y$ is a point at which $\varphi_1$ and $\varphi_2$ differ, necessarily $\varphi$ must differ at $y$ from either $\varphi_1$ or $\varphi_2$. Therefore,
$$\mu(C_1 \cup C_2 \cup \varphi_1 \Delta \varphi \cup \varphi \Delta \varphi_2) \geq \mu(C_1 \cup C_2 \cup \varphi_1 \Delta \varphi_2).$$
It follows that
$$\mu(C_1 \cup C_2 \cup \varphi_1 \Delta \varphi \cup \varphi \Delta \varphi_2) = \mu(C_1 \cup C_2 \cup \varphi_1 \Delta \varphi_2),$$
so that the equation \ref{eq:intervalleaf} provides
$$\sum\limits_{y \in Y} \left( d( \varphi_1(y), \varphi(y)) + d(\varphi(y),\varphi_2(y))- d(\varphi_1(y), \varphi_2(y)) \right)=0$$
Thus, for every $y \in Y$, the equality $d( \varphi_1(y), \varphi(y)) + d(\varphi(y),\varphi_2(y))= d(\varphi_1(y), \varphi_2(y))$ hods, which means that $\varphi(y) \in I(\varphi_1(y), \varphi_2(y))$. 

\medskip \noindent
Conversely, suppose that $\varphi(y) \in I( \varphi_1(y), \varphi_2(y))$ for every $y \in Y$. In particular, it implies that
$$\varphi_1 \Delta \varphi \cup \varphi \Delta \varphi_2 \subset \varphi_1 \Delta \varphi_2.$$
Indeed, if $\varphi_1$ and $\varphi_2$ agree at some $y \in Y$, then $\varphi(y) \in I(\varphi_1(y),\varphi_2(y))= \{ \varphi_1(y)=\varphi_2(y) \}$, so that $\varphi$ necessarily agrees with $\varphi_1$ and $\varphi_2$ at $y$. On the other hand, we already know that the converse inclusion holds (without any assumption), so we deduce that
$$\mu(C_1 \cup C_2 \cup \varphi_1 \Delta \varphi_2) = \mu(C_1 \cup C_2 \cup \varphi_1 \Delta \varphi \cup \varphi \Delta \varphi_2).$$
Because our assumption also implies that 
$$\sum\limits_{y \in Y} \left( d( \varphi_1(y), \varphi(y)) + d(\varphi(y),\varphi_2(y)) \right) = \sum\limits_{y \in Y} d(\varphi_1(y), \varphi_2(y)),$$
we conclude that the equation \ref{eq:intervalleaf} holds, and finally that the interval $I(x,y)$ intersects the leaf $\mathfrak{W}(\varphi)$.  
\end{proof}

\begin{proof}[Proof of Theorem \ref{thm:Wdeltamedian}.]
Let $x=(C_1,\varphi_1)$, $y=(C_2, \varphi_2)$ and $z=(C_3,\varphi_3)$ be three wreaths. Suppose that these three points of $\mathfrak{W}$ admit a median point $m=(C, \varphi) \in \mathfrak{W}$. It follows from Lemma \ref{lem:intervalleaf} that, for every $y \in Y$, $\varphi(y)$ belongs to $I(\varphi_1(y), \varphi_2(y)) \cap I(\varphi_2(y), \varphi_3(y)) \cap I(\varphi_1(y), \varphi_3(y))$, which means that $\varphi(y)$ is the median point of $\varphi_1(y)$, $\varphi_2(y)$ and $\varphi_3(y)$ in $X$. So $\varphi$ is uniquely determined. Next, because the interval $I(x,y)$ intersects the leaf $\mathfrak{W}(\varphi)$, we deduce from Fact \ref{fact:intervalleaf} that
$$\delta(x,y)= \delta(x,p_{\varphi}(x))+ \delta(p_{\varphi}(x),p_{\varphi}(y))+ \delta(p_{\varphi}(y),y).$$
On the other hand,
$$\delta(x,y)= \delta(x,m)+ \delta(m,y) = \delta(x,p_{\varphi}(x))+ \delta(p_{\varphi}(x),m)+ \delta(m,p_{\varphi}(y))+ \delta(p_{\varphi}(y),y).$$
Combining these two equalities yields
$$\delta(p_{\varphi}(x),p_{\varphi}(y))= \delta(p_{\varphi}(x),m)+ \delta(m,p_{\varphi}(y)).$$
We show similarly that
$$\delta(p_{\varphi}(x),p_{\varphi}(z))= \delta(p_{\varphi}(x),m)+ \delta(m,p_{\varphi}(z))$$
and
$$\delta(p_{\varphi}(y),p_{\varphi}(z))= \delta(p_{\varphi}(y),m)+ \delta(m,p_{\varphi}(z)).$$
Therefore, $m$ is also a median point of $p_{\varphi}(x)$, $p_{\varphi}(y)$ and $p_{\varphi}(z)$. Because the leaf $\mathfrak{W}(\varphi)$ is convex in $\mathfrak{W}$, according to Corollary \ref{cor:leafconvex}, and is a median space on its own right according to Proposition \ref{prop:Fmedian}, it follows that $p_{\varphi}(x)$, $p_{\varphi}(y)$ and $p_{\varphi}(z)$ admit a unique median point. Thus, we have proved that $x$, $y$ and $z$ admits at most one median point.

\medskip \noindent
Now, set $\varphi : y \mapsto m(\varphi_1(y), \varphi_2(y),\varphi_3(y))$ and let $m \in \mathfrak{W}(\varphi)$ denote the (unique) median point of $p_{\varphi}(x)$, $p_{\varphi}(y)$ and $p_{\varphi}(z)$. We want to prove that $m$ is a median point of $x$, $y$ and $z$. According to Lemma \ref{lem:intervalleaf}, the interval $I(x,y)$ intersects the leaf $\mathfrak{W}(\varphi)$, so that we deduce from Fact \ref{fact:intervalleaf} that
$$\begin{array}{lcl} \delta(x,y) & = & \delta(x,p_{\varphi}(x))+ \delta(p_{\varphi}(x),p_{\varphi}(y))+ \delta(p_{\varphi}(y),y) \\ \\ & =& \delta(x,p_{\varphi}(x))+ \delta(p_{\varphi}(x),m)+ \delta(m, p_{\varphi}(y))+ \delta(p_{\varphi}(y),y) \\ \\ & = & \delta(x,m)+ \delta(m,y) \end{array}$$
Similarly, we show that
$$\delta(x,z)=\delta(x,m)+\delta(m,z) \ \text{and} \ \delta(y,z)= \delta(y,m)+ \delta(m,z).$$
Thus, $m$ belongs to $I(x,y) \cap I(y,z) \cap I(x,z)$, ie., $m$ is a median point of $x$, $y$ and $z$.
\end{proof}

\begin{remark}
From the previous proof, we get a precise description of the median point $(M, \varphi)$ of three wreaths $(C_1, \varphi_1)$, $(C_2, \varphi_2)$ and $(C_3, \varphi_3)$. Indeed, 
$$\varphi : y \mapsto m( \varphi_1(y), \varphi_2(y), \varphi_3(y))$$
and $M$ is the convex hull of
$$\{ m(x_1,x_2,x_3) \mid x_i \in C_i \cup \varphi_i \Delta \varphi, \ i=1,2,3 \}.$$
\end{remark}

\subsection{Constructing median graphs}

\noindent
Let $X,Y$ be two median graphs and let $1 \in X$ be a basepoint. The distances between vertices of $X$ and $Y$ define two discrete median metrics, so that the distance $\delta$ on the diadem product $(X,1) \circledast Y$ turns out to be discrete as well, and median according to Theorem \ref{thm:Wdeltamedian}. Thus, $(X,1) \circledast Y$ can be thought of as a graph by linking any two points of $(X,1) \circledast Y$ at distance one appart by an edge, but does the resulting length metric coincide with $\delta$? The next lemma shows that this is the case, making $(X,1) \circledast Y$ a median graph. 

\begin{lemma}\label{lem:Wcube}
If $X$ and $Y$ are two median graphs, then $(X,1) \circledast Y$ is a median graph.
\end{lemma}

\begin{proof}
For short, we set $\mathfrak{W} = (X,1) \circledast Y$. Let $(C_1, \varphi_1),(C_2,\varphi_2) \in \mathfrak{W}$ be two wreaths. Define a sequence $R_1, \ldots , R_p \in \mathcal{F}(Y)$ of convex subcomplexes in the following way:
\begin{itemize}
	\item $R_1=C_1$;
	\item if $n \geq 2$ and $C_1 \cup C_2 \cup \varphi_1 \Delta \varphi_2 \subsetneq R_n$, $R_{n+1}$ is the convex hull of $R_n \cup \{x \}$, where $x$ is a vertex of the convex hull of $C_1 \cup C_2 \cup \varphi_1 \Delta \varphi_2$ which does not belong $R_n$ but which is adjacent to one of its vertices.
\end{itemize}
Notice that $(R_i,\varphi_1)$ and $(R_{i+1},\varphi_1)$ are at distance one appart in $\mathfrak{W}$ for every $1 \leq i \leq p-1$, and that $p= \# \mathcal{H}(C_1 \cup C_2 \cup \varphi_1 \Delta \varphi_2) \backslash \mathcal{H}(C_1)$. Similarly, define a sequence $S_1, \ldots, S_q \in \mathcal{F}(Y)$ from the convex hull of $C_1 \cup C_2 \cup \varphi_1 \Delta \varphi_2$ to $C_2$ such that $(S_i,\varphi_2)$ and $(S_{i+1},\varphi_2)$ are at distance one apart in $\mathfrak{W}$ for every $1 \leq i \leq q-1$ and such that $q= \# \mathcal{H}(C_1 \cup C_2 \cup \varphi_1 \Delta \varphi_2) \backslash \mathcal{H}(C_2)$. Finally, let $\psi_1, \ldots, \psi_r \in X^{(Y)}$ be a sequence of maps such that $\psi_1= \varphi_1$, $\psi_r=\varphi_2$, $s= \sum\limits_{y \in Y} d(\varphi_1(y),\varphi_2(y))$, and such that, for every $1 \leq i \leq r-1$, $\psi_i$ and $\psi_{i+1}$ differ at a single vertex $y$ and $\psi_i(y)$ and $\psi_{i+1}(y)$ are adjacent. Notice that $(R_p,\psi_i)$ and $(R_p,\psi_{i+1})$ are at distance one appart in the $\mathfrak{W}$ for every $1 \leq i \leq r-1$. Thus,
$$(R_1, \varphi_1), \ldots, (R_p,\varphi_1)=(R_p,\psi_1), \ldots, (R_p,\psi_r)=(S_1,\varphi_2), \ldots, (S_q, \varphi_2)$$
is a path in $\mathfrak{W}$, thought of as a graph, from $(C_1, \varphi_1)$ to $(C_2, \varphi_2)$ and of length
$$\# \mathcal{H}(C_1 \cup C_2 \cup \varphi_1 \Delta \varphi_2) \backslash \mathcal{H}(C_1) + \# \mathcal{H}(C_1 \cup C_2 \cup \varphi_1 \Delta \varphi_2) \backslash \mathcal{H}(C_1) + \sum\limits_{y \in Y} d(\varphi_1(y),\varphi_2(y)),$$
which is precisely the distance between $(C_1, \varphi_1)$ and $(C_2, \varphi_2)$. Consequently, the length distance on $\mathfrak{W}$ thought of as a graph coincides with $\delta$. Because we know from Theorem~\ref{thm:Wdeltamedian} that $\delta$ is a median distance, it follows that $\mathfrak{W}$ is a median graph. 
\end{proof}

\section{Coarse embeddings into $\ell^1$-spaces}

\noindent
In this section, our goal is to show that, if two finitely generated groups coarsely embed into median spaces, then we can combine these embeddings in order to coarsely embed the wreath product of our two groups into the diadem product of two corresponding median spaces. In fact, we will be able to work with graphs instead of groups thanks to the following definition:

\begin{definition}
Let $X,Y$ be two graphs and $1 \in X$ a basepoint. The \emph{wreath product} $(X,1) \wr Y$ is the graph whose vertices are the pairs $(\varphi,y)$ where $y \in Y$ and where $\varphi : Y \to X$ satisfies $\varphi(y)=1$ for all but finitely many $y \in Y$ (written $\varphi  \in X^{(Y)}$ in the sequel), and whose edges link two vertices $(\varphi_1,y_1), (\varphi_2,y_2)$ if either $\varphi_1=\varphi_2$ and $y_1,y_2$ are adjacent in $Y$ or $y_1=y_2$ and $\varphi_1,\varphi_2$ only differ at $y_1=y_2$ with $\varphi(y_1),\varphi(y_2)$ adjacent in $X$. 
\end{definition}

\noindent
Observe that, for all $(\varphi_1,y_1),(\varphi_2,y_2) \in (X,1) \wr Y$, one has
$$d((\varphi_1,y_1),(\varphi_2,y_2)) = \mathrm{TS}(y_1, \varphi_1 \Delta \varphi_2,y_2) + \sum\limits_{y \in Y} d(\varphi_1(y), \varphi_2(y))$$
where $\varphi_1 \Delta \varphi_2$ denotes the set of all points in $Y$ where $\varphi_1,\varphi_2$ differ and where $\mathrm{TS}(a,S,b)$ denotes the shortest length of a path that starts from a point $a$, that visits all the points in a set $S$, and that ends at a point $b$. Also, observe that, given two groups $G,H$ and two generating sets $R \subset G, S \subset H$, we have
$$\mathrm{Cayl}(G \wr H, R\cup S)= (\mathrm{Cayl}(G,R),1) \wr \mathrm{Cayl}(H,S),$$ 
justifying our terminology.

\begin{definition}
Let $G,H$ be two graphs, $X,Y$ two median spaces, $1 \in G$ a basepoint, and $\Phi : G \to X, \Psi : H \hookrightarrow Y$ two maps with $\Psi$ injective. The \emph{wreath product} $\Phi \wr \Psi : (G,1) \circledast H \to (X,\Phi(1)) \wr Y$ is
$$(c,h) \mapsto \left( \{\Psi(h)\}, y \mapsto \left\{ \begin{array}{cl} y & \text{if $y \notin \mathrm{Im}(\Psi)$} \\ \Phi(c(\Psi^{-1}(y))) & \text{otherwise} \end{array} \right. \right).$$
\end{definition}

\noindent
In the next subsections, we are going to show that, if $\Phi$ and $\Psi$ preserve the metrics of $G,H$, then so does $\Phi \wr \Psi$.

\subsection{Wreath products of coarse embeddings}

\noindent
Until the proof of Theorem \ref{thm:WreathCoarse} below, we fix two graphs $G,H$, a basepoint $1 \in G$, two median spaces $X,Y$, and two injective maps $\Phi : G \hookrightarrow X$, $\Psi : H \hookrightarrow Y$. Our goal is to prove the following statement:

\begin{prop}\label{prop:Coarse}
Assume that $H$ is uniformly locally finite. If $\Phi$ and $\Psi$ are coarse embeddings, then so is $\Phi \wr \Psi$.
\end{prop}

\noindent
We begin by proving two preliminary lemmas.

\begin{lemma}\label{lem:DistEstimation}
Ror all $(c_1,h_1),(c_2,h_2) \in (G,1) \wr H$,
$$\delta \left( \Phi \wr \Psi (c_1,h_1), \Phi \wr \Psi (c_2,h_2) \right) = 2 \mu \left( \Psi ( \{h_1,h_2\} \cup c_1 \Delta c_2) \right) + \sum\limits_{h \in H} d \left( \Phi(c_1(h)), \Phi(c_2(h)) \right).$$ 
\end{lemma}

\begin{proof}
For all $(c_1,h_1),(c_2,h_2) \in (G,1) \wr H$, we have
$$\sum\limits_{y \in \mathrm{Im}(\Psi)} d\left( \Phi(c_1(\Psi^{-1}(y))), \Phi(c_2(\Psi^{-1}(y))) \right) = \sum\limits_{h \in H} d \left( \Phi(c_1(h)), \Phi(c_2(h)) \right)$$
and
$$\begin{array}{lcl} \Phi c_1 \Psi^{-1} \Delta \Phi c_2 \Psi^{-1} & = & \left\{ y \in Y \mid \Phi \left( c_1 (\Psi^{-1}(y)) \right) \neq \Phi \left( c_2(\Psi^{-1}(y)) \right) \right\} \\ \\ & = & \left\{ y \in Y \mid c_1 \left( \Psi^{-1}y \right) \neq c_2 \left( \Psi^{-'1}(y) \right) \right\} = \Psi(c_1\Delta c_2)\end{array}$$
where we have denoted $\Phi c_i \Psi^{-1} : y \mapsto \left\{ \begin{array}{cl} y & \text{if $y \notin \mathrm{Im}(\Psi)$} \\ \Phi(c_i(\Psi^{-1}(y))) & \text{otherwise} \end{array} \right.$ for $i=1,2$ by abuse of notation. These two observations, applied to the definition of $\delta$, leads to the desired equality. 
\end{proof}

\begin{lemma}\label{lem:Lipschitz}
$\Phi \wr \Psi$ is Lipschitz.
\end{lemma}

\begin{proof}
Let $(c_1,h_1),(c_2,h_2) \in (G,1) \wr H$ be two adjacent vertices. Either $c_1=c_2$ and $h_1,h_2$ are adjacent in $H$, which implies according to Lemma \ref{lem:DistEstimation} that 
$$\delta \left( \Phi \wr \Psi (c_1,h_1), \Phi \wr \Psi (c_2,h_2) \right) = 2 d \left( \Psi(h_1), \Psi(h_2) \right) \leq C_1$$
for some uniform constant $C_1$; or $c_1,c_2$ differ only at $h_1=h_2$ and taking adjacent values, which implies according to Lemma \ref{lem:DistEstimation} that
$$\delta \left( \Phi \wr \Psi (c_1,h_1), \Phi \wr \Psi (c_2,h_2) \right) = d \left( \Phi(c_1(h_1)), \Phi(c_2(h_1)) \right) \leq C_2$$
for some uniform constant $C_2$. Therefore, we have
$$\delta \left( \Phi \wr \Psi (c_1,h_1), \Phi \wr \Psi (c_2,h_2) \right) \leq \min(C_1,C_2) \cdot d \left( (c_1,h_1),(c_2,h_2) \right)$$
for all $(c_1,h_1),(c_2,h_2) \in (G,1) \wr H$, concluding the proof of our lemma.
\end{proof}

\begin{proof}[Proof of Proposition \ref{prop:Coarse}.]
Assume that $\delta \left( \Phi \wr \Psi (c_1,h_1), \Phi \wr \Psi (c_2,h_2) \right) \leq R$ for some $R$. As a consequence of Lemma~\ref{lem:DistEstimation}, for all $(c_1,h_1),(c_2,h_2) \in (G,1)\wr H$, we have 
$$d(\Phi(c_1(h)),\Phi(c_2(h))) \leq R,$$ 
hence $d(c_1(h),c_2(h)) \leq C_1$ for some uniform constant $C_1$; we also have
$$\mathrm{diam}\left( \Psi(\{h_1,h_2\} \cup c_1 \Delta c_2) \right) \leq \mu\left( \Psi( \{h_1,h_2\} \cup c_1 \Delta c_2) \right) \leq R,$$
hence $\mathrm{diam}\left( \Psi(\{h_1,h_2\} \cup c_1 \Delta c_2) \right) \leq C_2$ for some uniform constant $C_2$. Consequently,
$$d((c_1,h_1),(c_2,h_2)) = \mathrm{TS}(h_1, c_1 \Delta c_2, h_2) + \sum\limits_{h \in c_1\Delta c_2} d(c_1(h),c_2(h))$$
is bounded above by a constant that depends only on $C_1$, $C_2$ and the maximal degree of a vertex in $H$. Together with Lemma \ref{lem:Lipschitz}, this concludes the proof that $\Phi \wr \Psi$ is a coarse embedding.
\end{proof}

\begin{proof}[Proof of Theorem \ref{thm:WreathCoarse}.]
Let $G,H$ be two graphs and $1 \in G$ a basepoint. If $G \wr H$ coarsely embeds in a Hilbert space, then so do $G,H$ since they isometrically embed in $G \wr H$. Conversely, assume that $G,H$ coarse embed in Hilbert spaces. It follows from \cite{MR2214573} that there exist two coarse embeddings $\Phi_0 : G \to X_0$ and $\Psi_0 : H \to Y_0$ in $\ell^1$-spaces, and so in median spaces. We can make $\Phi_0$ injective in the following way. Let $X$ denote the metric space obtained from $X_0$ by gluing, for every $x \in X_0$, the origins of $|\Phi_0^{-1}(x)|$ unit segments $[0,1]$ at $x$. Next, define $\Phi : G \to X$ in such a way that, for every $x \in X_0$, $\Phi$ sends the points in $\Phi_0^{-1}(x)$ to pairwise distinct endpoints of the new segments. Then $X$ is a median space containing $X_0$ as a convex subspace and $\Phi : G \hookrightarrow X$ is an injective coarse embedding. Similarly, we construct an injective coarse embedding to a median space $\Psi : H \hookrightarrow Y$ from $Y_0$ and $\Psi_0 : H \to Y$. We deduce from Proposition \ref{prop:Coarse} that $\Phi \wr \Psi$ defines a coarse embedding from $(G,1) \wr H$ to the median space $(X, \Phi(1)) \circledast Y$. As a median space always isometrically embeds in an $\ell^1$-space, which itself coarse embeds in an $\ell^2$-space, we conclude that $(G,1) \wr H$ coarsely embeds in a Hilbert space. 
\end{proof}

\subsection{A word about $\ell^1$-compressions}

\noindent
Until the proof of Theorem \ref{thm:Compression} below, we fix two graphs $G,H$, a basepoint $1 \in G$, two median spaces $X,Y$, and two injective coarse embeddings $\Phi : G \hookrightarrow X$, $\Psi : H \hookrightarrow Y$. Our goal is to prove the following statement:

\begin{prop}\label{prop:Compression}
Assume that there exist $\gamma,C>0$ such that
$$\mu \left( \{a,b\} \cup S \right) + |S| \geq C \cdot \mathrm{TS}(a,S,b)^\gamma \text{ for all $a,b \in Y$ and $S \subset Y$}.$$
If $\Phi, \Psi$ have respectively compressions $\geq \alpha, \geq \beta$, then $\Phi \wr \Psi$ has compression $\geq \gamma \min(\alpha, \beta)$. 
\end{prop}

\begin{proof}
We already know from Lemma \ref{lem:Lipschitz} that $\Phi \wr \Psi$ is Lipschitz. Let $\epsilon>0$ be smaller than the smallest distance between two distinct points in $\mathrm{Im}(\Psi)$. For convenience, we assume that $\epsilon \leq \min(2,2^\beta/C)$. Notice that, for all $a,b \in \mathrm{Im}(\Psi)$ and $S \subset \mathrm{Im}(\Psi)$ finite, we have
$$\begin{array}{lcl} 2 + \mu (\{a,b\} \cup S) + \epsilon|S| & \geq & C \epsilon \cdot \mathrm{TS}(a,S,b)^\beta \geq \frac{C \epsilon}{2^\beta} \left( \mathrm{TS}(a,S,b) + 2\epsilon |S| \right)^\beta \\ \\ & \geq & \frac{C \epsilon}{2^\beta} \left( \mathrm{TS}(a,S,b) + \epsilon |S| \right)^\beta \end{array}$$
where the second inequality is justified by $\mathrm{TS}(a,S,b) \geq (|S|+1) \epsilon \geq 2 \epsilon |S|$. 
According to Lemma~\ref{lem:DistEstimation} and the previous observation, for all $(c_1,h_1),(c_2,h_2) \in (G,1) \wr H$ we have
$$\delta \left( \Phi \wr \Psi (c_1,h_1), \Phi \wr \Psi (c_2,h_2) \right) =  2 \mu \left( \Psi ( \{h_1,h_2\} \cup c_1 \Delta c_2) \right) + \sum\limits_{h \in H} d \left( \Phi(c_1(h)), \Phi(c_2(h)) \right)$$
$$\begin{array}{cl} = & \displaystyle 2 \mu (\Psi(\{h_1,h_2\} \cup c_1 \Delta c_2)) + \epsilon |c_1 \Delta c_2| + \sum\limits_{h \in c_1\Delta c_2} \left( d(\Phi(c_1(h)),\Phi(c_2(h))) - \epsilon \right) \\ \\ \geq & \displaystyle \frac{C \epsilon}{2^\beta} \left( \mathrm{TS}(\Psi(h_1), \Psi(c_1 \Delta c_2), \Psi(h_2)) + \epsilon |S| \right)^\gamma + \sum\limits_{h \in c_1 \Delta c_2} \left( d(\Phi(c_1(h)),\Phi(c_2(h))) - \epsilon \right)^\gamma \\ \\  \geq & \displaystyle \frac{C\epsilon}{2^\beta} \left[ \mathrm{TS}(\Psi(h_1), \Psi(c_1\Delta c_2), \Psi(h_2)) + \sum\limits_{h\in c_1 \Delta c_2} d\left( \Phi(c_1(h)), \Phi(c_2(h)) \right) \right]^\gamma \\ \\ \geq & \displaystyle  \frac{C\epsilon}{2^\beta} \left[ C_2 \cdot \mathrm{TS}(h_1,c_1 \Delta c_2,h_2)^\beta+ C_1 \cdot \left( \sum\limits_{h \in c_1\Delta c_2} d(c_1(h),c_2(h)) \right)^\alpha \right]^\gamma \\ \\  \geq & \displaystyle \frac{C\epsilon}{2^\beta} \min(C_1,C_2)^\gamma \left[\mathrm{TS}(h_1,c_1 \Delta c_2,h_2)+  \sum\limits_{h \in c_1\Delta c_2} d(c_1(h),c_2(h))  \right]^{\gamma \min(\alpha,\beta)}  \\ \\  \geq & \displaystyle \frac{C\epsilon}{2^\beta} \min(C_1,C_2)^\gamma \cdot d((c_1,h_1),(c_2,h_2))^{\gamma \min(\alpha,\beta)} \end{array}$$
for some uniform constants $C_1,C_2>0$, proving that $\Phi \wr \Psi$ has compression $\geq \gamma \min(\alpha, \beta)$ as desired.
\end{proof}

\noindent
We denote by $\mathrm{TS}(Y)$ the supremum of the powers $\gamma$ such that $Y$ satisfies the condition mentioned in Proposition \ref{prop:Compression}. In \cite{Qm}, we investigated the possible values taken by $\mathrm{TS}(Y)$ when $Y$ a median graph. For instance, we proved that the following statements hold:
\begin{itemize}
	\item \cite[Lemma 9.44, Corollary 9.52]{Qm} If $Y$ is an unbounded median graph, then $\mathrm{TS}(Y)$ always belong to $[1/2,2/3] \cup \{1\}$.
	\item \cite[Proposition 9.45]{Qm} If $Y$ is a uniformly locally finite median graph, then $\mathrm{TS}(Y)=1$ if and only if $Y$ is hyperbolic.
	\item \cite[Corollary 9.53]{Qm} If $Y$ is a median graph containing a cube of arbitrary large dimension, then $\mathrm{TS}(Y)=1/2$.
	\item \cite[Lemma 9.50]{Qm} $\mathrm{TS}(\mathbb{Z}^d) = d/(2d-1)$ for every $d \geq 1$.
\end{itemize}
The uniform lower bound $\mathrm{TS}(Y)$ extends easily to the general case:

\begin{lemma}\label{lem:TShalf}
For all $a,b \in Y$ and $S \subset Y$ finite, we have
$$\mu(\{a,b\} \cup S ) + |S| \geq \mathrm{TS}(a,S,b)^{1/2}.$$
\end{lemma}

\begin{proof}
Fix an enumeration $S= \{s_1, \ldots, s_r\}$. Then
$$\begin{array}{lcl} \mathrm{TS}(a,S,b) & \leq & \displaystyle d(a,s_1) + \sum\limits_{i=1}^{r-1} d(s_i,s_{i+1}) + d(s_r,b) \leq (r+1) \mathrm{diam}(\{a,b\} \cup S) \\ \\ & \leq & 2|S| \mu(\{a,b\} \cup S) \leq (\mu(\{a,b\} \cup S) + |S|)^2.\end{array}$$
\end{proof}

\begin{proof}[Proof of Theorem \ref{thm:Compression}.]
If $\alpha_1(G)=0$ or $\alpha_1(H)=0$, there is nothing to prove, so from now on we assume that $\alpha_1(G), \alpha_1(H) \neq 0$. Fix an $\epsilon>0$ and a Lipschitz embedding $\Phi : G \to X$ (resp. $\Psi : H : Y$) to an $\ell^1$-space having compression $\geq \alpha_1(G)- \epsilon$ (resp. $\geq \alpha_1(H)-\epsilon$). Following the beginning of the proof of Theorem \ref{thm:WreathCoarse}, we can assume without loss of generality that $\Phi,\Psi$ are injective. We know from Proposition \ref{prop:Compression} and Lemma \ref{lem:Lipschitz} that $\Phi \wr \Psi$ is Lipschitz and has compression $\geq \mathrm{TS}(Y) \cdot \left(\min(\alpha_1(G),\alpha_1(H)) - \epsilon\right)$, and we know from Lemma \ref{lem:TShalf} that $\mathrm{TS}(Y) \geq 1/2$. Because every median space isometrically embeds in an $\ell^1$-space, it follows that there exists a Lipschitz embedding from $G \wr H$ to an $\ell^1$-space that has compression $\geq \left( \min(\alpha_1(G), \alpha_1(H))- \epsilon \right)/2$. We conclude the proof by letting $\epsilon \to 0$.
\end{proof}

\begin{proof}[Proof of Theorem \ref{thm:Hyperbolic}.]
Fix a biLipschitz embedding $\Phi : G \hookrightarrow X$ to an $\ell^1$-space and set $\Psi = \mathrm{id}_H$. According to \cite[Proposition 9.45]{Qm}, $\mathrm{TS}(Y)=1$. Therefore, Proposition~\ref{prop:Compression} and Lemma~\ref{lem:Lipschitz} imply that $\Phi \wr \Psi$ is a biLipschitz embedding to a median space. The desired conclusion follows from the fact that every median space isometrically embeds in an $\ell^1$-space.
\end{proof}

\section{Actions on $\ell^1$-spaces}\label{section:Actions}

\noindent
Fix two discrete groups $G,H$ respectively acting on two median spaces $X,Y$, with two points $x_0 \in X, y_0 \in Y$ having trivial stabilisers. Observe that the wreath product $G \wr H$ naturally acts on the diadem product $\mathfrak{W}:= (X,x_0) \circledast Y$ by isometries via
$$(h,\psi) \cdot (C, \varphi) = (hC, \overline{\psi}(\cdot) \varphi(h^{-1} \cdot)),$$
where $\overline{\psi} : Y \to G$ is defined by $\overline{\psi}(g \cdot y_0)=\psi(g)$ for every $g \in H$ and $\overline{\psi}(y)=1$ for every $y \notin H \cdot y_0$; if we view $H$ as a subset of $Y$ by taking its image under the orbit map associated to the basepoint $y_0$ (the orbit map being an embedding since $y_0$ has trivial stabiliser), then the map $\overline{\psi}$ is naturally an extension of $\psi$. It is straightforward to verify that this defines an isometric action of $G \wr H$ on $(\mathfrak{W},\delta)$. 

\medskip \noindent
In the next two sections, we show that $G \wr H \curvearrowright \mathfrak{W}$ inherits some properties from the actions $G \curvearrowright X$ and $H \curvearrowright Y$.

\subsection{Actions with unbounded orbits}

\noindent
First, we characterise when the action of the wreath product on the diadem product, as described above, has unbounded orbits.

\begin{prop}\label{prop:UnboundedOrbits}
Let $G,H$ be two non-trivial groups acting on two median spaces $X,Y$ with two points $x_0 \in X, y_0 \in Y$ having trivial stabilisers. If $G \cdot x_0$ is unbounded or if $\mu(H \cdot y_0)$ is infinite, then $G \wr H$ acts on $(X,x_0) \circledast Y$ with unbounded orbits.
\end{prop}

\begin{proof}
First, we observe that, if $G$ acts on $X$ with unbounded orbits, then $G$ (as the subgroup of $G \wr H$ indexed by $1 \in H$) also acts on $X \circledast Y$ with unbounded orbits. Indeed, if $\varphi : Y \to X$ denotes the map always taking the value $x_0$, then 
$$g \cdot (\{y_0\}, \varphi) = \left( \{y_0\} , \ y \mapsto \left\{ \begin{array}{cl} x_0 & \text{if $y \neq y_0$} \\ g x_0 & \text{otherwise} \end{array} \right. \right),$$
hence $\delta \left( g \cdot (\{y_0\},\varphi), (\{y_0\},\varphi) \right) \geq d(x_0,g\cdot x_0)$. The desired conclusion follows. 

\medskip \noindent
Next, we observe that, if $\mu(H \cdot y_0)$ is infinite, then $\bigoplus_H G$ acts on $X \circledast Y$ with unbounded orbits. Indeed, fix a finite subset $R \subset H$ and set $S:= \{h \cdot y_0 \mid h \in R\}$. Fix a non-trivial element $g \in G$ and let $\psi : H \to G$ denote the map that is identically equal to $g$ on $R$ and identically trivial elsewhere. Notice that
$$(1,\psi) \cdot (\{y_0\}, \varphi) = \left( \{y_0\}, \ y \mapsto \left\{ \begin{array}{cl} x_0 & \text{if $y \notin S$} \\ gx_0 & \text{otherwise} \end{array} \right. \right),$$
where $\varphi : Y \to X$ is identically to $x_0$, hence $\delta \left( (1,\psi) \cdot (\{y_0\}, \varphi), (\{y_0\}, \varphi) \right) \geq 2 \mu(S)$. Because $\mu(H \cdot y_0)$ is infinite, we can choose $R$ so that $\mu(S)$ is arbitrarily large, so the desired conclusion follows.
\end{proof}

\noindent
Theorem \ref{thm:TandHaagerup} essentially follows from the combination of \cite{medianviewpoint} and Proposition \ref{prop:UnboundedOrbits}. The only point to be careful with is that our construction start with actions on median spaces having basepoints with trivial stabilisers. However, it essentially follows from \cite[Lemma 4.34]{Qm} that the assumption is not restrictive. For completeness, we reproduce the argument below.

\begin{lemma}\label{lem:modifycubing}
Let $G$ be a group acting on a median space $X_0$. Then $G$ acts on a median space $X$ containing $X_0$ so that the action $G \curvearrowright X_0$ extends to an action $G \curvearrowright X$ and $X$ contains a vertex whose stabiliser is trivial. Moreover, the action $G \curvearrowright X$ is properly discontinuous (resp. metrically proper, cocompact, with unbounded orbits) if and only if the action $G \curvearrowright X_0$ is properly discontinuous (resp. metrically proper, cocompact, with unbounded orbits) as well. 
\end{lemma}

\begin{proof}
Let $x_0 \in X_0$ be a base vertex and let $\Omega$ denote its $G$-orbit. Let $X$ be the space constructed from $X_0$ by adding one point $(x,g)$ for every $x \in \Omega$ and $g \in \mathrm{stab}(x)$, and one segment of length one between $x$ and $(x,g)$ for every $x \in \Omega$ and $g \in \mathrm{stab}(x)$. It is straightforward to verify that $X$ is a median space. 

\medskip \noindent
Now, we extend the action $G \curvearrowright X_0$ to an action $G \curvearrowright X$. For every $x \in \Omega$, fix some $h_x \in G$ such that $h_x \cdot x_0=x$. For every $g,k \in G$ and $x \in \Omega$, define
$$g \cdot (x,k)= (gx,gk h_x h_{gx}^{-1});$$
notice that
$$gkh_x h_{gx}^{-1} \cdot gx= gkh_x \cdot x_0 = gk \cdot x = g \cdot x,$$
so that $gkh_xh_{gx}^{-1} \in \mathrm{stab}(gx)$. Moreover, 
$$\begin{array}{lcl} g_1 \cdot ( g_2 \cdot (x,k)) & = & g_1 \cdot (g_2x,g_2kh_xh_{g_2x}^{-1}) \\ \\ & = & (g_1g_2x, g_1 \cdot g_2kh_xh_{g_2x}^{-1} \cdot h_{g_2x} h_{g_1g_2x}^{-1}) \\ \\ & = & (g_1g_2x, g_1g_2 k h_x h_{g_1g_2x}^{-1}) = g_1g_2 \cdot (x,k) \end{array}$$
so we have defined a group action $G \curvearrowright X$, which extends $G \curvearrowright X_0$ by construction. 

\medskip \noindent
Fixing some $x \in \Omega$, we claim that the vertex $(x,1) \in X$ has trivial stabiliser. Indeed, if $g \in G$ fixes $(x,1)$, then $(x,1)= g \cdot (x,1)= (gx, gh_xh_{gx}^{-1})$. As a consequence, $gx=x$, ie., $g \in \mathrm{stab}(x)$, so that $h_{gx}=h_x$. Therefore, our relation becomes $(x,1)=(x,g)$, hence $g=1$. 

\medskip \noindent
This proves the first assertion of our lemma. Next, it is clear that the action $G \curvearrowright X$ is properly discontinuous (resp. metrically proper, cocompact, with unbounded orbits) if and only if the action $G \curvearrowright X_0$ is properly discontinuous (resp. metrically proper, cocompact, with unbounded orbits) as well.
\end{proof}

\begin{proof}[Proofs of the first parts of Theorems \ref{thm:TandHaagerup} and \ref{thm:FWandPW}.]
Let $H$ act on the tree $T_H$ whose vertex-set is $H \cup \{H\}$ and whose edges connect every $h \in H$ to $H$. The vertex $1 \in T$ has trivial stabiliser and $\mu(H \cdot 1) = |H|$. Similarly, let $G$ act on the tree $T_G$ constructed in the same way. If $H$ is infinite, it follows from Proposition \ref{prop:UnboundedOrbits} that $G \wr H$ acts on the median space $(T_G,1) \circledast T_H$ with unbounded orbits. Therefore, $G \wr H$ does not have property (T). If $G$ does not have property (T), then according to \cite{medianviewpoint} it admits an action on a median space $X$ with unbounded orbits. According to Lemma \ref{lem:modifycubing}, we can suppose without loss of generality that $X$ contains a point $x_0$ with trivial stabiliser. We conclude from Proposition \ref{prop:UnboundedOrbits} that $G \wr H$ acts on $(X,x_0) \circledast T_H$ with unbounded orbits, and consequently that $G \wr H$ does have property (T). Conversely, if $H$ is finite, then $G \wr H$ contains a finite-index subgroup  isomorphic to a product of finitely many copies of $G$, so it follows from basic properties satisfied by (T) that $G \wr H$ has property (T) if so does $G$ (see for instance \cite{MR2415834}). 

\medskip \noindent
Thus, we have proved the first part of Theorem \ref{thm:TandHaagerup}. As a consequence of Lemma \ref{lem:Wcube}, reproducing the same argument word for word proves the first part of Theorem \ref{thm:FWandPW}.
\end{proof}

\subsection{Proper actions}

\noindent
Finally, we characterise when the action of the wreath product on the diadem product, as described at the beginning of~Section \ref{section:Actions}, is metrically proper.

\begin{prop}\label{prop:properaction}
Let $G,H$ be two discrete groups acting on two median spaces $X,Y$ with two points $x_0 \in X, y_0 \in Y$ having trivial stabilisers. If the actions $G \curvearrowright X$ and $H \curvearrowright Y$ are metrically proper, then so is the action of $G \wr H$ on $\mathfrak{W}:= (X,x_0) \circledast Y$.
\end{prop}

\begin{proof}
It is sufficient to prove that, fixing some $R \geq 0$, the set
$$F = \left\{ (h, \psi) \in G \wr H \mid \delta((h, \psi) \cdot (\{y_0 \}, \xi),(\{y_0\}, \xi)) \leq R \right\}$$
is finite, where $\xi$ denotes the map $Y \to X$ constant to $x_0$. So, by definition of $F$, an element $(h,\psi)\in G \wr H$ belongs to $F$ if and only if
$$2 \cdot \mu \left( \{y_0,hy_0\} \cup \overline{\psi} \xi \Delta \xi \right) + \sum\limits_{g \in H} d(\psi(g) \cdot x_0,x_0) \leq R.$$
If $(h,\psi)$ is such an element, in particular
$$d(y_0,hy_0) = \mu \left( \{ y_0,hy_0\} \right) \leq R,$$
and since the action $H \curvearrowright Y$ is metrically proper, it follows that $h$ can take only finitely many values. Moreover, if we denote by $\mathrm{supp}(\psi)$ the set $\{ g \in H \mid \psi(h) \neq 1 \}$, notice that $\overline{\psi} \xi \Delta \xi$ coincides with $\mathrm{supp}(\psi) \cdot y_0$. Consequently, 
$$d(y_0,sy_0) \leq \mu \left( \{y_0,hy_0\} \cup \overline{\psi} \xi \Delta \xi \right) \leq R$$
for every $s \in \mathrm{supp}(\psi)$, so that, once again because the action $H \curvearrowright Y$ is metrically proper, there are only finitely many choices for $\mathrm{supp}(\psi)$. Finally, notice that, for every $k \in \mathrm{supp}(\psi)$,
$$d(\psi(k) \cdot x_0, x_0) \leq \sum\limits_{g \in H} d(\psi(g) \cdot x_0,x_0) \leq R,$$
so that, because the action $G \curvearrowright X$ is metrically proper, $\psi(k)$ can take only finitely many values. Thus, we have proved that there are only finitely many choices on $h$ and $\psi$ in order to have $(h,\psi) \in F$. A fortiori, $F$ must be finite.
\end{proof}

\begin{proof}[Proofs of the second parts of Theorems \ref{thm:TandHaagerup} and \ref{thm:FWandPW}.]
If $G \wr H$ is a-T-menable, then clearly $G$ and $H$ are also a-T-menable. Conversely, assume that $G$ and $H$ are a-T-menable. According to \cite{medianviewpoint}, $G$ (resp. $H$) acts metrically properly on a median space $X$ (resp. $Y$); as a consequence of Lemma \ref{lem:modifycubing}, we can assume that $X$ (resp. $Y$) contains a point $x_0$ (resp. $y_0$) with trivial stabiliser. It follows from Proposition \ref{prop:properaction} that $G \wr H$ admits a metrically proper action on a median space, and we conclude from \cite{medianviewpoint} that the wreath product is a-T-menable. 

\medskip \noindent
Thus, we have proved the second part of Theorem \ref{thm:TandHaagerup}. As a consequence of Lemma \ref{lem:Wcube}, reproducing the same argument word for word proves the second part of Theorem~\ref{thm:FWandPW}.
\end{proof}

\noindent
We conclude this article by noticing that, in the context of median graphs, we are also able to construct properly discontinuous actions

\begin{thm}\label{thm:propercube}
If $G$ and $H$ are two groups acting properly discontinuously on some median graphs, then their wreath product $G \wr H$ acts properly discontinuously on a median graph as well.
\end{thm}

\begin{proof}
Let $G$ and $H$ act properly discontinuously on median graphs $X$ and $Y$ respectively. By following Lemma \ref{lem:modifycubing} (or according to \cite[Lemma 4.34]{Qm}), we can suppose without loss of generality that there exist vertices $x_0 \in X$ and $y_0 \in Y$ with trivial stabilisers. We deduce from Lemma \ref{lem:Wcube} that the wreath product $G\wr H$ acts on the median graph $\mathfrak{W}:= (X,x_0) \circledast Y$. We claim that this action is properly discontinuous, which amounts to saying that vertex-stabilisers of $\mathfrak{W}$ are finite.

\medskip \noindent
So let $(C, \varphi) \in \mathfrak{W}$ be a wreath. An element $(h,\psi) \in G \wr H$ belongs to its stabiliser if and only if
$$(C, \varphi)= (h,\psi) \cdot (C, \varphi)= (hC, \overline{\psi}(\cdot)\varphi(h^{-1} \cdot)),$$
ie., $hC=C$ and $\overline{\psi}(\cdot)\varphi(h^{-1} \cdot)= \varphi(\cdot)$. In a median graph, the convex hull of a finite set must be finite, so that, because the action $H \curvearrowright Y$ is properly discontinuous, there may exist only finitely many $h \in H$ satisfying $hC=C$. From now on, suppose that $h \in H$ is fixed, and satisfies $hC=C$. Notice that the condition $\overline{\psi}(\cdot)\varphi(h^{-1} \cdot)= \varphi(\cdot)$ implies that $\psi(g) \cdot \varphi(h^{-1}g \cdot y_0) = \varphi(g \cdot y_0)$ for every $g \in G$. As a consequence, if we set $F= \{ g \in G \mid \varphi(g \cdot y_0) \neq x_0 \}$, then, for every $g \notin F \cup h F$, one has $\psi(g) \cdot x_0=x_0$, so that $\psi(g)=1$ since the stabiliser of $x_0$ is trivial. On the other hand, $F$ is finite because
$$F \subset \bigcup \left\{ \mathrm{stab}_G(y) \mid \varphi(y) \neq x_0 \right\}$$
and because the action $G \curvearrowright X$ is properly discontinuous, so we have only finitely many choices for $\mathrm{supp}(\psi) = \{ g \in G \mid \psi(g) \neq 1 \}$. If $g \in F \cup hF$, then there exist some $y_1,y_2 \in \Phi := \varphi(F \cup hF)$ such that $\psi(g) \cdot y_1=y_2$; since $\Phi$ is finite and that the action $G \curvearrowright X$ is properly discontinuous, we deduce that we have only finitely many choices for $\psi(g)$. Thus, we have proved that there exist only finitely many $h \in H$ and $\psi \in G^H$ such that $(h,\psi)$ belongs to the stabiliser of $(C, \varphi)$, which precisely means that this stabiliser must be finite. This concludes the proof.
\end{proof}

\addcontentsline{toc}{section}{References}

\bibliographystyle{alpha}
{\footnotesize\bibliography{WreathMedian}}

\Address

\end{document}